\documentclass[12pt]{amsart}
\usepackage{amssymb,amscd,amsmath,amsthm,color}
\usepackage{calc}
\usepackage{amsrefs}
\usepackage{color}
\usepackage[T1]{fontenc}
\usepackage{mathtools}
\usepackage[normalem]{ulem}
\usepackage{tikz}
\usetikzlibrary{matrix,arrows,decorations.pathmorphing}
\usepackage{setspace}
\usepackage{verbatim}
\usepackage{mathrsfs}

\definecolor{bettergreen}{rgb}{0,.7,0}

\newcommand\todo[1]{\ \vspace{5mm}\par \noindent\framebox{\begin{minipage}[c]{0.95 \textwidth} \tt #1\end{minipage}} \vspace{5mm} \par}

\DeclarePairedDelimiter\ceil{\lceil}{\rceil}
\DeclarePairedDelimiter\floor{\lfloor}{\rfloor}

\newcommand{\Q}{{\mathbb Q}}

\newcommand{\R}{{\mathbb R}}
\newcommand{\Z}{{\mathbb Z}}

\newcommand{\CC}{{\mathbb C}}
\newcommand{\ZZ}{{\mathbb Z}}

\newcommand{\SL}{\mathbf {SL}}
\newcommand{\Sp}{\mathbf {Sp}}

\newcommand{\bG}{\mathbf{G}}
\newcommand{\bT}{\mathbf {T}}

\newcommand{\cF}{\mathcal{F}}
\newcommand{\cO}{\mathcal{O}}

\newcommand{\cB}{\mathcal{B}}

\newcommand{\fg}{\mathfrak{g}}
\newcommand{\ft}{\mathfrak{t}}

\newcommand{\fsp}{\mathfrak{sp}}
\newcommand{\fsl}{\mathfrak{sl}}
\newcommand{\fO}{\mathfrak{O}}
\newcommand{\fP}{\mathfrak{P}}

\newcommand{\fh}{\mathfrak{h}}
\newcommand{\fF}{\mathfrak{F}}

\newcommand{\reg}{\mathrm{rss}}

\newcommand{\nil}{\mathrm{Nil}}
\newcommand{\val}{\mathrm{val}}
\newcommand{\rs}{\mathrm{rs}}

\newcommand{\diag}{\mathrm{diag}}

\newcommand{\Ad}{\operatorname{Ad}}

\newcommand{\cI}{{\mathcal I}}

\newcommand{\rf}{k}

\newcommand{\ldp}{{\mathcal L}_{\mathrm {DP}}}

\newcommand\cA{{\mathcal A}}
\newcommand\ord{\mathrm{ord}}
\newcommand\ac{\overline{\mathrm{ac}}}

\newcommand\cC{{\mathcal C}}
\newcommand\cH{{\mathcal H}}

\newcommand{\rss}{\mathrm{rss}}

\newcommand{\K}{F}

\newcommand{\tf}{{C_c^\infty}}

\newcommand{\Hasse}{\operatorname{Hasse}}
 \newcommand{\spa}{\operatorname{span}}

 \newcommand{\disc}{\operatorname{disc}}
 
 \def\ve{\varepsilon}
 \def\an{_{\text{aniso}}}
 \def\gerg{\mathfrak{g}}

\def\R{\mathbb R}

\def\Q{\mathbb Q}
\def\Z{\mathbb Z}

\def\GL{\mathrm{GL}}

\def\cO{\mathcal{O}}
\def\cQbar{\overline{\mathcal{Q}}}

\def\sp{\mathfrak{sp}}

\def\cN{\mathcal{N}}
\def\cE{\mathcal{E}}
\def\cF{\mathcal{F}}
\def\cH{\mathcal{H}}

\def\cQ{\mathcal{Q}}

\def\gO{\mathfrak{O}}

\def\gP{\mathfrak{P}}

\def\inv{^{-1}}

\theoremstyle{plain}
\newtheorem{thm}{Theorem}
\newtheorem{theorem}[thm]{Theorem}
\newtheorem{lem}[thm]{Lemma}

\newtheorem{prop}[thm]{Proposition}

\theoremstyle{definition}
\newtheorem{rem}[thm]{Remark}
\newtheorem{defn}[thm]{Definition}
\newtheorem{example}[thm]{Example}

\title[]{Shalika germs for $\fsl_n$ and $\fsp_{2n}$ are motivic}
\author{Sharon Frechette, Julia Gordon and Lance Robson}

\begin{document}

\maketitle

\begin{center}
\today
\end{center}


\begin{abstract}
We prove that Shalika germs on the Lie algebras $\fsl_n$ and $\fsp_{2n}$ belong to the class of so-called ``motivic functions'' defined by means of a first-order language of logic. 
It is a well-known theorem of Harish-Chandra that for a Lie algebra $\fg(F)$ over a local field $F$ of characteristic zero, the  Shalika germs, normalized by the square root of the absolute value of the discriminant, are bounded on the set of regular semisimple elements 
$\fg^{\rss}$, however it is not easy to see how this bound depends on the field $F$. 
As a consequence of the fact that Shalika germs are motivic functions for $\fsl_n$ and $\fsp_{2n}$, we prove that for these Lie algebras, this bound 
must be of the form $q^a$, where $q$ is the cardinality of the residue field of $F$, and $a$ is a constant.   
Our proof that Shalika germs are motivic in these cases relies on the interplay of DeBacker's parametrization of nilpotent orbits with the parametrization using partitions, and the explicit matching between these parametrizations due to M.Nevins, \cite{nevins:param}. We include two detailed examples of the matching of these parametrizations.
\end{abstract}


\section{Introduction}
In this paper we prove that Shalika germs for the Lie algebras of type $\fsl_n$ and 
$\fsp_{2n}$ belong to the class of so-called motivic functions,  and explore some of the consequences of this fact.

Shalika germs first appeared in the papers of Shalika \cite{shalika:germs} and 
Harish-Chandra \cite{hc:harmonic-williamstown}. The survey of their role in harmonic 
analysis on $p$-adic groups is beyond the scope of this paper; we refer the reader to 
the beautiful article by Kottwitz \cite{kottwitz:clay}, and to \cite{hc:queens} for the detailed definitions and main results regarding them. We simply note that Shalika germs, by definition, are functions on the set of regular semisimple elements in a Lie algebra, yet except for those defined on a few Lie algebras of small rank, their exact values elude computation. 
Here we use a general theorem about uniform bounds for motivic functions proved in \cite{S-T}*{Appendix B} to 
estimate the absolute values of the Shalika germs in a uniform way over all local fields of a given (sufficiently large) residue characteristic. 




First, let us recall the definitions.  Let $F$ be a local field, $\bG$ be a connected reductive algebraic group over $F$, and $\fg$ its Lie algebra.  
In our results, $\bG=\SL_n$ or $\Sp_{2n}$, although several of the background results hold in greater generality.
Let $X\in \fg(F)$, with adjoint orbit  $\cO_X=\{\Ad(g)X\mid g\in \bG(F)\}$ and stabilizer $C_G(X)$. (Since here we are dealing with the classical Lie algebras, the Adjoint action is just conjugation: $\Ad(g)X=gXg^{-1}$.) 
The space $\cO_X$ with the $p$-adic topology is homeomorphic to $\bG(F)/C_G(X)$, which carries a $G$-invariant quotient measure. 
For the fields $\K$ of characteristic zero, it was proved by Deligne and Ranga Rao \cite{ranga-rao:orbital} that when transported to the orbit of $X$,  this measure  is a Radon measure on $\fg(\K)$, {\em i.e.}, it is finite on compact subsets of $\fg(\K)$. 
(Strictly speaking, it is the group version of this statement that is proved in \cite{ranga-rao:orbital}, but in characteristic zero this is equivalent to the Lie algebra version.) 
Denote this quotient measure on $\bG(\K)/C_G(X)$ by $d^{\ast}g$. 
The \emph{orbital integral} at $X$ is  the distribution $\mu_X$ on $\tf(\fg(\K))$ defined by
\begin{equation}\label{eqn:orbital}
\mu_X(f)=\int_{\bG(\K)/C_G(X)}f(\Ad(g) X)d^\ast g.
\end{equation}
For the fields of sufficiently large positive characteristic (with an explicit bound on the characteristic), convergence of orbital integrals was proved by McNinch, \cite{mcninch:nilpotent}. 

There are finitely many nilpotent orbits in $\fg(F)$, provided the field $F$ has characteristic zero or sufficiently large positive characteristic (depending on the root system of $\fg$). 
The Shalika germ expansion expresses the regular semisimple orbital integrals as linear combinations of nilpotent ones, in a neighbourhood of the origin. 
More precisely, let $\nil(F)$ denote the finite set of nilpotent orbits in $\fg(F)$,
let $\fg(F)^\reg$ denote the set of regular semisimple elements in $\fg(F)$, 
 and for each $\cO\in \nil(F)$ let $\mu_{\cO}$ be the orbital integral over $\cO$ (it is a linear functional on $\tf(\fg(F))$).
For every $f\in \tf(\fg(F))$ there exists a neighbourhood $U_f$ of zero in $\fg(F)$, and functions $\Gamma_{\cO}(X)$ defined on 
$\fg(F)^\reg\cap U_f$, such that  for all 
$X\in \fg(F)^\reg\cap U_f$, we have the expansion
\begin{equation}\label{eq:Sh_germs}
\mu_X(f)=\sum_{\cO\in \nil(F)} \Gamma_{\cO}(X)\mu_{\cO}(f).
\end{equation}
The functions $\Gamma_{\cO}$ are called \emph{provisional Shalika germs}, using the terminology of \cite{kottwitz:clay}*{\S6, \S17}.
These provisional Shalika germs 
are well-defined as germs of functions at the origin; moreover they possess a natural homogeneity property, and using this they can be extended canonically to the entire set  $\fg(F)^\rss$ (see \cite{kottwitz:clay}*{\S 17} for details).  

The goal of this paper is to prove that provisional Shalika germs belong to the class of the so-called motivic functions.
This was proved for $\fg=\fsp_{2n}$ by  L. Robson in his M.Sc. essay \cite{lance:thesis}; we include this case here since it was not published elsewhere. 
We also study the case $\fg=\fsl_n$ which is in some ways simpler, but has a technical issue that does not arise in the $\fsp_{2n}$ case (namely, the dependence of the set of nilpotent orbits on the field $F$); this was the content of our WIN project. We present both cases in detail here in preparation for a general proof for all Lie algebras, which will appear elsewhere.  

The class of motivic functions was defined by R. Cluckers and F. Loeser in \cite{cluckers-loeser}. 
{Concretely}, motivic functions are complex-valued functions on $p$-adic manifolds defined uniformly in $p$
by means of a first-order language of logic, called Denef-Pas language, which we will define below. We include a brief,  simplified version of the definition of motivic functions, but for the details, as well as a survey of the applications of this class of functions in harmonic analysis on $p$-adic groups, we refer the reader to the survey 
\cite{CGH-ad} and the original papers \cite{cluckers-loeser}, \cite{cluckers-hales-loeser}. The aim of this paper is to add Shalika germs to the list of functions arising in harmonic analysis that can be studied via motivic integration techniques, 
in the case of $\fg=\fsl_n$ or $\fsp_{2n}$.

Cluckers, Hales and Loeser prove in \cite{cluckers-hales-loeser} that regular semisimple orbital integrals are motivic, and in \cite{CGH-2} the same statement is proved  for all, and in particular, nilpotent orbital integrals. 
Thus, once we have shown that the functions $\Gamma_{\cO}$ are motivic, we see that both sides of (\ref{eq:Sh_germs}) are motivic functions. 
As an immediate consequence of the Transfer Principle proved by Cluckers and Loeser \cite{cluckers-loeser}, this shows the Shalika germ expansion holds for fields of sufficiently large positive characteristic. (This was previously proved by DeBacker \cite{debacker:homogeneity}; our results give an alternative proof.) 
More importantly, the uniform boundedness result from \cite{S-T}*{Appendix B} then implies the uniform bound on Shalika germs normalized by the square root of the discriminant (see Theorem \ref{thm:bound} below).

Our main results are stated and proved in \S\ref{sec:main}.  The rest of the paper provides a review of all the prerequisites, thus experts may want to turn immediately to the last section. The proof that provisional Shalika germs are motivic functions has two main ingredients: first we must establish a way to describe nilpotent orbits in the motivic context, and second, we find definable test functions which allow us to isolate individual Shalika germs in the linear combination and therefore show that they are motivic. The first step requires a parametrization of nilpotent orbits that is as field-independent as possible, and this is where partitions are advantageous. For the second step, it is convenient to use DeBacker's parametrization of orbits. The proof of the main theorem essentially works in much greater generality than stated here, except we do not quite have the structure in Denef-Pas language that would capture the set of nilpotent orbits generally. Here, for the special cases of $\fsl_n$ and $\fsp_{2n}$, we use the matching between the two parametrizations of nilpotent orbits that was established by Nevins \cite{nevins:param}. Since all three of the authors found this material challenging to absorb, in \S\ref{sec:match} we include detailed examples for $\fsl_3$ and $\fsp_4$; we hope they will be useful for future students.  

{\bf Acknowledgement.} This paper clearly owes a debt to the ideas of T.C. Hales and to the thesis of Jyotsna Diwadkar. The second author is grateful to Raf Cluckers and Immanuel Halupczok for multiple helpful communications. We thank the organizers of the WIN workshop in Luminy who made this collaboration possible. The second and third authors were supported by NSERC. 


\section{Motivic functions}\label{sec:motfun}
This section is included in order for the paper to be self-contained. 
However, this overview of the definitions has appeared in various forms in several papers on the topic; the present version is quoted nearly verbatim from \cite{CGH-ad}, except for \S \ref{sect:ldpo}, which is new and specifically adapted for the purposes of this paper.  

Informally, motivic functions are built from definable functions in the Denef-Pas language. Thus they are given independently of the field and can be interpreted in any non-Archimedean local field. 
We first recall the definition of the Denef-Pas language. 


\subsection{Denef-Pas language}\label{sub:DP}
Denef-Pas language is a first order language of logic designed for working
with valued fields. Formulas in this language will allow us to uniformly handle sets and functions for all local fields.
We start by defining two sublanguages of the language of
Denef-Pas: the language of rings and  Presburger language.
\subsubsection{The language of rings}

Apart from the symbols for variables $x_1, \dots, x_n,\dots$ and the
usual logical symbols equality `$=$', parentheses `$($', `$)$', the quantifiers `$\exists$', `$\forall$', and the logical operations conjunction `$\wedge$', negation `$\neg$', disjunction `$\vee$', the language of rings consists of the following symbols:
\begin{itemize}
\item constants `$0$', `$1$';
\item binary functions `$\times$', `$+$'.
\end{itemize}

A (first-order) formula in the language of rings is any syntactically correct formula
built out of these symbols. (One usually omits the words `first order'.)  If a formula in the language of rings has $n$ free 
variables, then it defines a subset of $R^n$ for any ring $R$.
For example the formula ``$\exists x_2\, (x_2\times x_1 = 1)$'' defines the set of units $R^\times$ in any ring $R$.  Note that by convention, quantifiers always run over the ring in question.
Note also that quantifier-free formulas in the language of rings define constructible sets, as they appear in classical algebraic geometry.


\subsubsection{Presburger language}\label{Pres}
A formula in Presburger language is built out of variables running over $\Z$, the logical symbols (as above) and
symbols `$+$', `$\le$', `$0$', `$1$', and for each $d=2,3,4,\dots$, a symbol `$\equiv_d$' to denote the binary
relation $x\equiv y \pmod{d}$.
Note the absence of the symbol for multiplication.


\subsubsection{Denef-Pas language}\label{DP}
The Denef-Pas language is a three-sorted language in the sense that its formulas utilize
three different ``sorts'' of elements: those of the valued field, of the residue field, and of the value group (which will always be $\Z$ in our setting).
Each variable in such a formula runs over only the elements of one of the sorts,
so there are three disjoint sets of symbols for the variables of the different sorts.
To create a syntactically-correct formula, one must pay attention to the sorts when composing functions
and inserting them into relations.

In addition to the variables and the logical symbols, the formulas use the following symbols:
\begin{itemize}
\item In the valued field sort:
the language of rings.
\item In the residue field sort: the language of rings.
\item In the $\Z$-sort: the Presburger language.
\item {the} symbol $\ord(\cdot)$ for the valuation map from the nonzero elements of the valued field sort to the $\Z$-sort, and {the} symbol $\ac(\cdot)$ for the so-called angular component, which is a {multiplicative} function from the valued field sort to the residue field sort (more about this function below).
\end{itemize}

A formula in this language can be interpreted in any discretely valued field $F$ which comes with a
uniformizing element $\varpi$,
by letting the variables range over $F$, {over its} residue field $k_F$, and {over} $\Z$,
respectively, depending on the sort to which they belong;
$\ord$ is the valuation map (defined on $F^\times$ and such that $\ord(\varpi)=1$), and $\ac$ is defined as follows:
if $x$ is a unit (that is, $\ord(x)=0$), then $\ac(x)$ is the residue of $x$ modulo $\varpi$ (thus, an element of the residue field); for all other
nonzero $x$, one puts $\ac(x) :=  \varpi^{-\ord(x)}x \bmod (\varpi)$.  Thus, for $x\neq 0$, $\ac(x)$ is the residue class of the 
first non-zero coefficient of the $\varpi$-adic expansion of $x$. Finally, we define $\ac(0)=0$.

Thus, a formula $\varphi$ in this language with $n$ free valued-field variables, $m$ free residue-field variables, and $r$ free $\Z$-variables gives naturally, for each discretely valued field $F$, a subset $\varphi(F)$ of
$F^n\times \rf_F^m\times \Z^r$: namely, $\varphi(F)$ is the set of all the tuples for which the interpretation of $\varphi$ in $F$ is ``true''.

We will denote this language by $\ldp$. 


\subsection{Definable sets and motivic  functions}\label{subsub:functions}

The $\ldp$-formulas introduced in the previous section allow us to obtain a field-independent notion of subsets of $F^n\times k_F^m\times \ZZ^r$ for all local fields $F$ of sufficiently large residue characteristic. 
The reason behind the restriction on characteristic is explained below in 
Remark \ref{rem:largep}. 
 
\begin{defn}\label{defset}
A collection
$X = (X_F)_{F}$ of subsets $X_F\subset F^n\times \rf_F^m\times \Z^r$ is called a \emph{definable set} if there is an
$\ldp$-formula $\varphi$ {and an integer $M$} such that $X_F = \varphi(F)$ for each 
$F$ {with residue characteristic at least $M$ (cf. Remark \ref{rem:largep})}, where $\varphi(F)$ is as described at the end of \S \ref{DP}.
\end{defn}

By Definition \ref{defset}, a definable set is actually a collection of sets indexed by non-Archimedean local fields $F$; such practice is not uncommon in model theory and has its analogues in classical algebraic geometry.
A particularly simple definable set is $(F^n\times k_F^m\times \ZZ^r)_F$, for which  we introduce
the simplified notation
${\rm VF}^n\times {\rm RF}^m\times \ZZ^r$, where ${\rm{VF}}$ stands for valued field and ${\rm{RF}}$ for residue field.
We apply the typical set-theoretical notation to definable sets $X, Y$,    {\em e.g.},
$X \subset Y$ (if $X_F \subset Y_F$ for each $F$), $X \times Y$, and so on.

\begin{defn}\label{deffunct}
For definable sets $X$ and $Y$, a collection $f = (f_F)_F$ of functions $f_F:X_F\to Y_F$  is called a {\em definable function} and denoted by $f:X\to Y$ if
the collection of graphs of the $f_F$ is a definable set.
\end{defn}

\begin{rem}\label{rem:largep}
There is a subtle issue here, due to the fact that the same definable set can be defined by different formulas. Technically, {it would be more elegant to think of a definable set} as an equivalence class of what we have called definable sets in Definition \ref{defset}, where we call two such definable sets equivalent if they are the same for all $F$ with sufficiently large residue characteristic.
To ease notation, we will not emphasize this point, but because of this 
all results presented in this paper will only be valid {\em for fields with sufficiently large residue characteristic}.  In particular, we assume hereafter that char$(F) \neq 2$.
\end{rem}

We now come to motivic  functions, for which definable functions are the building blocks.
{We note that while definable functions, by definition, must be ${\rm VF}^n\times {\rm RF}^m\times \ZZ^r$-valued for some $m,n,r$, the 
\emph{motivic} functions are built from {definable sets and  functions}, and can be thought of as complex-valued functions (although here they will naturally be $\Q$-valued).  This does not require thinking of rational or complex numbers in the context of logic; these are just usual complex-valued functions that happen to be built from definable ingredients as prescribed by the following definition.
 
\begin{defn}\label{motfun}
Let $X = (X_F)_F$ be a definable set.
A collection $f = (f_F)_F$ of functions $f_F:X_F\to\CC$ is called \emph{a motivic function} on $X$ if and only if
there exist integers
$N$, $N'$, and $N''$, such that, for all non-Archimedean local fields $F$, 
\begin{equation}\label{eqn:motivic_fcn}
f_F(x)=\sum_{i=1}^N   q_F^{\alpha_{iF}(x)} ( \# (Y_{iF})_x  )  \bigg( \prod_{j=1}^{N'} a_{ijF}(x) \bigg) \bigg( \prod_{\ell=1}^{N''} \frac{1}{1-q_F^{a_{i\ell}}} \bigg), \mbox{ for } x\in X_F,
\end{equation}
for some
\begin{itemize}
\item nonzero integers $a_{i\ell}$, 
\item definable functions $\alpha_{i}:X\to \ZZ$ and $\beta_{ij}:X\to \ZZ$, \item definable sets  $Y_i\subset X\times {\rm RF}^{r_i}$,
\end{itemize}
where, for $x\in X_F$,  $(Y_{iF})_x$ is the finite set $\{y\in k_F^{r_i}\mid (x,y)\in Y_{iF}\}$, and $q_F$ is the cardinality of the residue field $k_F$.

We call a motivic function on a one-point set a \emph{motivic constant}.
\end{defn}

In Theorem \ref{thm:bound}, we will need to allow the square root  of the cardinality of the residue field as a possible value of a motivic function. Hence, we will use the slightly generalized notion of a motivic function introduced in \cite{CGH-2}*{\S B.3.1}. Namely, given an integer $r>0$ and a definable $\Z$-valued function $f$, 
expressions of the form $q_F^{f/r}h$, where $h$ is a motivic function as above, will also  be called motivic functions.

\subsection{Adding constants to the language}\label{sect:ldpo}
We will need to extend Denef-Pas language by adding finitely many constant symbols in the valued field sort, whose role will be to encode units whose angular components form a set of representatives of $k_F^\times/(k_F^\times)^m$,  where $m$ is a fixed integer. 
Such extensions were first used by T.C. Hales, and an extension very similar to the one we define here appears first in J. Diwadkar's thesis, \cite{diwadkar:thesis}*{\S 2.2.3}.

Specifically, let $m$ be a fixed integer. We add $m$ constant symbols $d_1, \dots, d_m$ to the valued  field sort of Denef-Pas language, to obtain the language ${\ldp}_m$. 

Now we need to define their interpretation, given a local field $F$ with a uniformizer $\varpi$ and residue field $k$. 

If the set $k^\times/(k^\times)^m$ has $m$ elements, then we want
$d_1, \dots, d_m$ to be interpreted as units such that their angular components form a set of representatives of distinct $(k^\times)^m$-cosets. 
Specifically, we can write a formula 
$$\exists y_1, \dots, y_m \in F^\times,  \ord(y_i)=0, \nexists z: y_i = y_jz^m \text{ if } i\neq j.$$
This formula is true for $F$ under our assumption. Then we can set the values of $d_1, \dots, d_m$ in $F$ to be any collection $\{y_1, \dots, y_m\}$ satisfying this formula. 

If the cardinality of $k_F^\times/(k_F^\times)^m$ is  equal to $\ell< m$, we write a similar formula 
stating that $\{y_1, \dots , y_\ell\}$ are  
distinct representatives of $(k_F^\times)^m$-cosets. 
More precisely,
for every divisor $\ell$ of $m$, 
let $\phi_{\ell,m}$ be the following formula, with the quantifiers ranging over the residue field sort: 
\begin{equation}\label{eq:philm}
\phi_{\ell, m}(y_1, \dots, y_{\ell}):=`\nexists z: \
 y_i=y_j z^m \text{ for } i\neq j \wedge \forall x \exists z, x=y_i z^m \text { for some }1\le i \le \ell.\text{'}
 \end{equation}
(This formula is written slightly informally; in reality it contains a conjunction of 
$\ell(\ell-1)/2$ formulas, and a disjunction of $\ell$ formulas.) 
This formula states that $y_1, \dots ,y_\ell$ are distinct representatives of 
$k_F^\times/ (k_F^\times)^m$ in $k_F^\times$.

For a given finite field $k$ and fixed $m$, exactly one of the statements
$$
\psi_{\ell, m}:=
`\exists y_1, \dots, y_{\ell},\,  \phi_{\ell,m}(y_1, \dots , y_{\ell})\text{'} 
$$ holds, as 
$\ell$ runs over all divisors of $m$. 
If $\psi_{\ell, m}$ holds in $k_F$, we interpret the constant symbols  $d_{1},
\dots, d_{\ell}$ as units  of the valued field such that 
$\phi_{\ell, m}(\ac(d_{1}), \dots, \ac(d_{\ell}))$ holds.
Set the rest of the $d_{i}$ equal to $1$. 

All the constructions and theorems of motivic integration do not change if we add finitely many constant symbols. Hereafter, we fix an integer $n$ (coming from a fixed Lie algebra $\fsl_n$ or $\fsp_{2n}$), and say that a set or function is 
\emph{definable} if it is definable in the language ${\ldp}_m$ for some $m\le n P(n)$, where $P(n)$ is the number of partitions of $n$. We shall see later that we may need to consider the union of languages ${\ldp}_m$ as $m$ varies over a set of integers associated with partitions of $n$; however, it does not matter how many constants we add, as long as it is a finite number that is fixed in advance.

In the same way that a non-Archimedean local field $F$ with a choice of the uniformizer is a structure for the language $\ldp$, we note that a structure for ${\ldp}_m$ is a non-Archimedean local field $F$ with a choice of the uniformizer of the valuation, and a choice of a collection of units whose angular components form a set of representatives of $k_F^\times/(k_F^\times)^m$.

With this terminology, we can now state this paper's goal precisely: to show that Shalika germs are motivic functions, up to dividing by a motivic constant, in the sense that they are motivic functions where we use the language ${\ldp}_m$ with some finite $m$. 
We note that not all motivic constants are invertible in the ring of motivic functions, which is why we require the ``up to motivic constant'' provision.


\section{Classification of nilpotent orbits of $\fsl_n$ and $\fsp_{2n}$ via partitions.}

As discussed in the Introduction, we 
study two parameterizations of nilpotent orbits in $\fsl_n$ and in $\fsp_{2n}$, with a view toward defining these orbits by formulas in Denef-Pas language.
In this section we recall a well-known parametrization involving partitions,
and in \S \ref{sec:debacker_param} we recall a parametrization due to DeBacker \cite{debacker:nilp}, involving the Bruhat-Tits building for $\fg$.
In fact, a proof of definability of the nilpotent orbits for $\fsl_n$ using DeBacker's parametrization is carried out explicitly in Diwadkar's thesis \cite{diwadkar:thesis}.
Here we recast it in a slightly simpler form, taking advantage of the explicit matching between the 
two parametrizations, as proved by Nevins \cite{nevins:param}, and also of recent developments in the theory of motivic integration that allow us to  slightly simplify Diwadkar's terminology. 
\subsection{Notation}
Hereafter, $F$ will stand for a non-Archimedean local field with $\mathrm{char}(F)~\neq~2$, and $\overline{F}$ for a separable closure of $F$.
The ring of integers of $F$ will be denoted by 
$\fO$ (or $\fO_F$ if there is a possibility of confusion), the maximal ideal by $\fP$,  and the residue field by $k_F$.
 We will always assume that $F$ comes with a choice of the uniformizer of the valuation $\varpi$, and when talking about the language ${\ldp}_m$, with a choice of representatives for $\fO^\times/(\fO^\times)^m$ as discussed above in \S \ref{sect:ldpo}. 

\subsection{Parametrization of nilpotent orbits in $\fsl(n)$ using partitions}
For a positive integer $n$, a {\em partition} $\lambda = (\lambda_1, \lambda_2, \ldots, \lambda_t)$ of $n$ is a weakly-decreasing sequence of positive integers whose sum is $n$.  We say the $\lambda_i$ are the {\em parts} of the partition $\lambda$, and the {\em length} of $\lambda$ is $t$.  For each $1 \leq j \leq n$, the {\em multiplicity} $m_j(\lambda)$ is the number of parts of $\lambda$ satisfying $\lambda_i = j$.  We denote the greatest common divisor of the parts $\lambda_i$ by $\gcd(\lambda)$.

It is well known that when the characteristic of $F$ is greater than $n$, the set of nilpotent orbits of $\fsl_n(\overline{F})$ is in one-to-one correspondence with the set of partitions of $n$ (see \cite{collingwood-mcgovern} or \cite{waldspurger:nilpotent}, for instance).  A nilpotent orbit corresponds to the partition whose parts are determined by the blocks in its Jordan normal form.  Specifically, let $\lambda = (\lambda_1, \lambda_2, \ldots, \lambda_t)$ be a partition of $n$, and let $J_{\lambda_i}$ denote the $\lambda_i\times \lambda_i$-matrix whose 
$(j, j+1)$ entries are equal to $1$ for $1\le j\le \lambda_i$, with all remaining entries equal to $0$.  Let $J_\lambda$ denote the $n\times n$-matrix in Jordan normal form whose Jordan blocks are the $J_{\lambda_i}$, and let $\cO_\lambda$ denote the nilpotent orbit in $\fsl_n(\overline{F})$ with representative~$J_\lambda$. 


The explicit correspondence between partitions and $F$-rational nilpotent orbits is described in the following proposition.  The number of $F$-rational nilpotent orbits depends both on the partition $\lambda$ and on the characteristic of $F$, in a controlled way.

\begin{prop}(\cite{nevins:param}, Prop 4)\label{prop:partitions}
Let $\lambda$ be a partition of $n$, and $m~=~\gcd(\lambda)$.  For any $d \in F^\times$ define the $n\times n$-matrix $D(d) = \diag(1, 1, \ldots, 1, d)$.
	\begin{enumerate}
	\item For each $d \in F^\times$, the matrix $X_d = J_\lambda D(d)$ represents a $F$-rational orbit in $\cO_\lambda(F)$, and conversely every orbit has a representative of this form.
	\item The $\SL_n(F)$-orbits represented by $J_\lambda D(d)$ and $J_{\lambda'}D(d')$ coincide if and only if $\lambda = \lambda'$ and $d \equiv d'$ in $F^\times/ (F^\times)^m$.
	\end{enumerate}
\end{prop} 

\begin{example}\label{example:sl3_partitions}
In the case of $\fsl_3$, we have three partitions: $\lambda = (3)$, $(2,1)$, and $(1,1,1)$. The corresponding nilpotent orbits $\cO_\lambda$ in $\fsl_3(\overline{F})$ have representatives\\
$X_{(3)} = \begin{pmatrix}
	0 & 1 & 0\\
	0 & 0 & 1\\
	0 & 0 & 0
	\end{pmatrix}$,
$X_{(2,1)} =\begin{pmatrix}
	0 & 1 & 0\\
	0 & 0 & 0\\
	0 & 0 & 0
	\end{pmatrix}$, and
$X_{(1,1,1)} =\begin{pmatrix}
	0 & 0 & 0\\
	0 & 0 & 0\\
	0 & 0 & 0
	\end{pmatrix}$ respectively.
	
The nilpotent orbits $\cO_{(2,1)}$ and $\cO_{(1,1,1)}$ do not split further into distinct $F$-rational orbits, since $m = \gcd(\lambda) = 1$ for these partitions.   Since $m=3$ for the first partition, the nilpotent orbit $\cO_{(3)}$ splits into $|F^\times/(F^\times)^3|$ distinct $F$-rational orbits, represented by the matrices 
	$$X_d = J_{(3)}D(d) = \begin{pmatrix}
	0 & 1 & 0\\
	0 & 0 & d\\
	0 & 0 & 0
	\end{pmatrix},$$ 
one for each distinct equivalence class of $d$ in $F^\times/(F^\times)^3$.    By our assumptions, $F$ has residue characteristic $\neq 2$ and its residue field $k_F$ has $q= p^k$ elements, where $p$ is prime.  By standard results in group theory, the number of cubes in $k_F^\times$ is $\frac{q-1}{\gcd(3, q-1)}$, and so the cardinality of $k_F^\times/ (k_F^\times)^3$ is $\gcd(3, q-1)$.  Thus the number of distinct $F$-rational orbits in this case is
$$\left| F^\times/(F^\times)^3 \right| = 3 \cdot \left| k_F^\times/ (k_F^\times)^3 \right| 
= \begin{cases}
9 & \text{ if $3 \mid (q-1)$,}\\
3 & \text{ otherwise}.
\end{cases}$$
\bigbreak

\end{example}


\subsection{Parametrization of nilpotent orbits in $\fsp_{2n}$ using partitions}\label{subsec:sp_2n-nilp}

In the case of $\fsp_{2n}$, classes of quadratic forms over $F$ take the place of the cosets 
$F^\times/(F^\times)^m$ that we have seen in the parametrization of nilpotent orbits in the $\fsl_n$ case.  Thus, we begin by recalling the classification of quadratic forms.


\subsubsection{Quadratic forms}
Let $V$ be a finite-dimensional vector space over $F$, and $Q$ a non-degenerate quadratic form defined on $V$.  
Recall that the quadratic space $(V,Q)$ over $F$ is {\em anisotropic} if there is no nonzero $\mathbf{x} \in V$ such that $Q(\mathbf{x}) = 0$, and is {\em isotropic} otherwise.  

Consider the quadratic form $q_0: F^2 \rightarrow F$ that is represented in the standard basis by the matrix $q_0 = \begin{pmatrix} 0 & 1\\1 & 0 \end{pmatrix}$.  
The quadratic space $(F^2, q_0)$ is the {\em hyperbolic plane}, a key example in the theory of quadratic forms.  
Since char$(F) > 2$, if  $(V,Q)$ is a non-degenerate quadratic space over $F$, then by the Witt decomposition (see \cite{lam:qforms} for instance), the quadratic form $Q$ can be decomposed into an orthogonal direct sum
\begin{equation}\label{eqn:witt-decomp}
Q = q_0^m \oplus Q\an,
\end{equation}
 for some $m \leq \frac{1}{2} \dim(Q)$, where $(V\an,Q\an)$ is anisotropic and uniquely determined up to isometry. The integer $m$ is called the {\em Witt index} of
$(V,Q)$ and the quadratic form $Q\an$ is called the {\em anisotropic part} of $Q$.  Moreover, quadratic forms of a given dimension may be classified by their discriminant and Hasse invariant.  

Since char$(F) \neq 2$, we
have $F^\times / (F^\times)^2 \simeq \Z/2\Z\times \Z/2\/\Z$; thus
there are at most 8 nondegenerate quadratic forms over $F$ of a given
dimension. 
On the other hand, 
the maximum possible dimension of an
anisotropic form over $F$ is four.  Thus, by the Witt decomposition, to list
all equivalence classes of quadratic forms, it suffices to list the
classes of anisotropic forms.  Representatives for these classes are given in the following lemma.

\begin{lem}\cite{nevins:param}*{Lemma 3}\label{anisoList}
Let $F$ be as above.  If $-1\in (F^\times)^2$, let $\alpha= \ve$ be a fixed nonsquare unit in
$F$. If $-1\notin (F^\times)^2$, let $\alpha = 1$ and $\ve = -1$.  Given a
quadratic form $Q$, its anisotropic part $Q\an$ is either the zero
subspace, or isometric to one of the 15 anisotropic
forms in the following table.
\end{lem}

\begin{table}[htbp]
\centering
\begin{tabular}{|c|c|clc|}
\hline
Dimension & $\disc(Q)$ & $\Hasse(Q)$ & Representative & \\
\hline
1 & 1 & 1 & 1 & \\
1 & $\ve$ & 1 & $\ve$ & \\
1 & $\varpi$ & 1 & $\varpi$ & \\
1 & $\ve\varpi$ & 1 & $\ve\varpi$ & \\
\hline
2 & $\alpha $ & 1 & $\diag(1,\alpha )$ & \\
2 & $\alpha $ & $-1$ & $\diag(\varpi,\alpha \varpi)$ & \\
2 & $tt^\prime\varpi$ & $(t,t^\prime\varpi)_F$ &
$\diag(t,t^\prime\varpi)$ & $t,t^\prime \in \{1,\ve\}$ \\
\hline
3 & $t$ & $-1$ & $\diag(\alpha  t, \varpi, \alpha \varpi)$ & $t\in\{1,\ve\}$ \\
3 & $\alpha  t \varpi$ & $(\alpha , \varpi)_F$ & $\diag(1,\alpha, t\varpi)$ &
$t\in\{1,\ve\}$ \\
\hline
4 & 1 & $-1$ & $\diag(1,-\ve,-\varpi,\ve\varpi)$ & \\
\hline
\end{tabular}
\medbreak
\caption{Explicit representatives of the 15 nonzero equivalence
  classes of anisotropic quadratic forms over a local $F$ with residue characteristic not 2. 
  }
\label{table}
\end{table}

Given a class of quadratic forms, a matrix representative of the
form $Q=q_0^m\oplus Q\an$, where $Q\an$ is one of the diagonal matrices
given in the above table, will be called a {\em minimal matrix
representative} of the class.


\subsubsection{Partition parametrization of the nilpotent orbits in
   $\mathfrak{sp}_{2n}$}\label{qfParam} 
Embed $\Sp_{2n}$ into $\GL_{2n}$ as $\Sp_{2n} = \{g\in \GL_{2n}\,:\, g^t J g = J\}$, where
$J=\bigl( \begin{smallmatrix} 0 & I \\ -I &
  0\end{smallmatrix}\bigr)$. 
Note that this is a different embedding than the one used by Waldspurger \cite{waldspurger:nilpotent}, however, the parametrization below follows his methods. Let $V$ denote the vector space of the natural representation of $\fsp_{2n}$, with symplectic form defined by $\langle x, y\rangle = x^t J y$.

The nilpotent Adjoint orbits in $\fsp_{2n}(\overline{F})$ are parametrized by partitions 
$\lambda$ of
$2n$ in which the odd parts have even multiplicity (\cite{collingwood-mcgovern}, Corollary 4.1.8). 
For such a partition $\lambda$, let $\cO_\lambda$ denote the geometric nilpotent orbit corresponding to $\lambda$. 

The $F$-points, $\cO_\lambda(F)$, of this orbit may fail to be a single
$\Sp_{2n}(F)$-orbit, so the set of partitions $\lambda$ is no longer sufficient to parametrize the nilpotent orbits over $F$.
Instead, there is a set defined in terms of classes of quadratic forms
corresponding to the partition $\lambda$ that parametrizes the $\Sp_{2n}(F)$-orbits in
$\cO_\lambda(F)$. Let $\cQbar = (\cQ_2,\ldots,\cQ_{2n})$ be an $n$-tuple of
isometry classes of quadratic forms over $F$. We say that $\cQbar$
{\em corresponds to the partition} $\lambda$ of $2n$ (whose odd parts have even
multiplicities) if $\dim(\cQ_i) = m_i(\lambda)$ for each $i=2,\ldots,2n$.

\begin{thm}(\cite{nevins:param}*{Proposition 5}, due to Waldspurger \cite{waldspurger:nilpotent})\label{thm:sp4-partitions}
Let $\lambda$ be a partition of $2n$, and suppose the odd parts of $\lambda$ have even multiplicity. 
Then $\cO_\lambda(F)$
is a union of $\Sp_{2n}(F)$-orbits parametrized by the $n$-tuples 
$$\cQbar = (\cQ_2,\ldots,\cQ_{2n})$$ 
corresponding to $\lambda$ (as defined above), where $\cQ_i$ is an isometry class of a nondegenerate quadratic form over
$F$.
\end{thm}

Following Nevins \cite{nevins:param}, for each pair $(\lambda,\cQbar)$ we give an explicitly defined $X\in \fsp_{2n}(F)$ in the corresponding nilpotent orbit. 
We first give a decomposition of the vector space $V$ corresponding to the partition $\lambda$, and then define $X$ by its action on each component.

Let $\{p_1,\ldots,p_n,q_1,\ldots,q_n\}$ denote a {\em symplectic basis} for $V$; 
that is, a basis such that $\langle p_i, q_j\rangle = \delta_{ij}$, $\langle q_i, p_j\rangle = -\delta_{ij}$, and $\langle
p_i, p_j\rangle = \langle q_i, q_j\rangle = 0$. 
For each $i\in \{1,\ldots,2n\}$, let $s_i = \displaystyle \sum_{j<i}\frac{1}{2}\, j\, m_j$. 
Then the elements $s_i$ are integers such that $0 = s_1 \leq s_2 \leq \ldots \leq s_{2n} \leq n$. 
For each $j$ with $m_j\neq 0$, let $V(j)$ be the subspace given by
\begin{equation}\label{basis} 
V(j) = \spa\{p_{s_j+1},\ldots, p_{s_j+\frac{1}{2}jm_j}, q_{s_j+1},\ldots, q_{s_j+\frac{1}{2}jm_j}\}.
\end{equation}
Then $V = \displaystyle \bigoplus_{j: m_j\neq 0} V(j)$, so we may define $X$ by its action on each subspace $V(j)$.


If $j$ is odd, let $\mu = (j,\ldots, j)$, a partition of $\frac{1}{2}\,j\,m_j$, and define the restriction
of $X$ to $V(j)$ with respect to the basis given in (\ref{basis}) by
\begin{equation}\label{jOdd} 
X|_{V(j)} 
	= \begin{pmatrix} J_\mu & 0 \\ 0 & -J_\mu^t \end{pmatrix}. 
\end{equation}
If $j=2N$ is even, define $X|_{V(j)}$  with respect to the
basis given in (\ref{basis}) by
\begin{equation}\label{jEven} 
X|_{V(j)} 
	= \begin{pmatrix} J_{Nm_j}^{m_j} & Z\oplus (-1)^N Q_j \\ 0 & -(J_{Nm_j}^{m_j})^t \end{pmatrix},
\end{equation}
where  $Z$ is the $m_j(N-1) \times m_j(N-1)$ zero
matrix and $Q_j$ is the minimal matrix representative of $\cQ_j$.
Then we have the following correspondence:
\begin{thm}[\cite{nevins:param}, adapted from Proposition 6]\label{reps}
Let $\lambda$ be as above. The matrix $X\in \fsp_{2n}(F)$ defined by (\ref{jOdd})
and (\ref{jEven}) is a representative of the $\Sp_{2n}(F)$-orbit in
$\cO_\lambda(F)$ corresponding to the $n$-tuple $\cQbar$.
\end{thm}




\section{Parametrization of nilpotent orbits via the building}\label{sec:debacker_param}


\subsection{Preliminaries regarding the building}

Following the notation and terminology of \cite{nevins:param}, we briefly recall the necessary facts about the standard affine apartment of the Bruhat-Tits building $\cB(\bG)~=~\cB(\bG, F)$
for $\bG$ a connected reductive algebraic group over $F$.
However, since this is 
the only case we need in this paper, we assume that $\bG$ is split over $F$, which simplifies these definitions substantially.

Let $\bT$ be a split maximal torus of $\bG$.  
Let $X^\ast(\bT)$ be the group of $F$-rational characters of $\bT$ and let $X_\ast(\bT)$ be the group of $F$-rational cocharacters.  
Let $\langle \, , \, \rangle: X^\ast(\bT) \times X_\ast(\bT) \rightarrow \Z$ denote the natural pairing.  
Let $\Phi = \Phi(\bG, \bT)$ denote the set of roots of $\bT$ in $\bG$; it is  a finite subset of $X^\ast(\bT)$, and $\fg$ has the root space decomposition
\begin{equation*}\label{eqn:fg-decomp}
\fg = \ft \oplus \bigoplus_{\alpha \in \Phi} \fg_\alpha,
\end{equation*} where $\ft$ is the Lie algebra of $\bT$ and the root subspace $\fg_\alpha$ is defined by 
\begin{equation*}\label{fg_alpha}
\fg_\alpha = \{ X \in \fg \mid \Ad(t)X = \alpha(t)X \text{ for all } t \in \bT\}.
\end{equation*}
The standard {\em affine apartment} $\cA$ in the building $\cB(\bG)$ is the affine space underlying the vector space $X_\ast(\bT) \otimes_\Z \R$, together with a hyperplane structure.

Let $W(\Phi)= W(\bG,\bT)$ denote the Weyl group of $\bT$ in $\bG$.  The Weyl group is generated by reflections through hyperplanes corresponding to each root $\alpha \in \Phi$; its action on $X^\ast(\bT)$ preserves $\Phi$. 

For each $\alpha \in \Phi$ and $n \in \Z$, we consider the affine functional, or {\em affine root}, $\alpha+n: \cA \rightarrow \R$ defined for each $x = \lambda \otimes s \in \cA$ by $$(\alpha + n)(x) = \langle \alpha + n, \lambda \otimes x \rangle = s\langle \alpha, \lambda \rangle + n.$$  Put $\Psi = \{ \alpha+n \mid \alpha \in \Phi, n \in \Z\},$ and for each $\psi = \alpha + n \in \Psi$ consider the hyperplane 
$$H_{\psi} = \{ x \in \cA \mid \psi(x) = 0\}.$$
The set of all such hyperplanes forms a hyperplane structure on $\cA$.


\subsubsection{The standard apartment for $\fsl(n)$}
When $\bG = \SL_n$, let $\bT$ be the diagonal torus and consider the apartment $\cA$ corresponding to $\bT$.    
We identify $X_{\ast}(\bT)$ with $\Z^n$, and think of cocharacters explicitly as functions $t\mapsto \diag(t^{x_1}, \dots, t^{x_n})$, for $t\in F^\times$ and $(x_1, \dots, x_n)\in \Z^n$. 
 Identify $\cA$ with $X_\ast(\bT) \otimes \R$, and for each $i$ with $1 \leq i \leq n$, define the mapping $e_i:\cA \rightarrow \R$ by $e_i(f \otimes s) = sx_i$, for 
 $f = (t\mapsto \diag(t^{x_1}, t^{x_2}, \ldots, t^{x_n})) \in X_\ast(\bT)$ and 
 $s \in \R$.   Then the set of roots is given by 
\begin{equation}\label{eqn:sln-roots}
\Phi = \{e_i - e_j \mid 1 \leq i \neq j \leq n\}.
\end{equation}  Each of the root spaces $\fg_{e_i - e_j}$ is one-dimensional.  We may view $\fg_{e_i - e_j}$ as being spanned by the matrix $E_{ij}$ whose entries are all zero except for the $(i,j)$-entry, which equals 1.



\subsubsection{The standard apartment for $\fsp_{2n}$}
When $\bG = \Sp_{2n}$, again let $\bT$ be the diagonal torus, whose elements are of the form $\tau = \diag(t_1, t_2, \ldots, t_n, t_1^{-1}, t_2^{-1}, \ldots, t_n^{-1})$ by our choice of the embedding.  


The rank of $\mathfrak{sp}_{2n}$ is $n$, so we have $X^*(\bT)\simeq X_*(\bT)\simeq \Z^n$, as abelian groups. 
Similar to the $\fsl_n$ case, for $1\leq i \leq n$, define the mapping $e_i:\cA \rightarrow \R$ by
$e_i(f\otimes s) = s x_i$, for 
$f = (t\mapsto \diag(t_1^{x_1}, t_2^{x_2}, \ldots, t_n^{x_n}, t_1^{-x_1}, t_2^{-x_2}, \ldots, t_n^{-x_n}))$ and 
$x \in \R$.
Then $\Phi$ is the set
\begin{equation}\label{eqn:sp2n-roots}
\Phi = \{ e_i-e_j,\, \pm(e_i+e_j),\,\pm2e_i \, | \,
1\leq i\neq j \leq n\}.
\end{equation}
Finally, let $\cA$ be the standard apartment
of $\mathfrak{sp}_{2n}$ relative to this root datum.

Below in Sections \ref{example:sl3} and \ref{example:sp4}, we discuss in detail the examples of $\fsl(3)$ and $\fsp(4)$. 


\subsection{DeBacker's parametrization using the building}


Generalizing the work of Barbasch and Moy \cite{barbasch-moy:loc-char-exp}, DeBacker \cite{debacker:nilp} developed a parametrization of nilpotent orbits that relies on facets in the Bruhat-Tits building of $\fg$ and is valid for any reductive Lie algebra over $F$, provided the residue characteristic is sufficiently large.  This is actually a family of parametrizations that depends on a real parameter $r$, although for our purposes it suffices to consider the case corresponding to $r=0$, in the notation of \cite{debacker:nilp}.  To ease notation, we omit the $r$-dependence in DeBacker's notation and state suitably modified versions of the relevant theorems with $r=0$.    

Let $\cA$ denote the standard affine apartment corresponding to the Lie algebra $\gerg$. The set $\cA$ has the structure of a simplicial complex (generally,  polysimplicial, but the groups we are considering in this paper are simple). 
Let us define the {\em facets} ({\em i.e.} the simplices) in the apartment $\cA$.
For $x\in \cA$ and $n\in \ZZ$, define the sets
\begin{equation}\label{eqn:phi-H}
\Phi_x = \{\alpha \in \Phi \mid \alpha(x) \in \ZZ\} \qquad \text{ and } \qquad \cH_n = \{x \in \cA \mid \, |\Phi_x| = n\}.
\end{equation}
For an integer $n$, a {\em facet} of $\cA$ is defined to be any connected component $\cF$ of $\cH_n$.  We denote by $A(\cF, \cA)$ the smallest affine subspace of $\cA$ containing $\cF$.  With this, we define the dimension of a facet to be $\dim(\cF) = \dim A(\cF, \cA)$, hence facets of the apartment $\cA$ have bounded dimension. 

Given a subspace $H$ of $\cA$, a facet $\cF \subset H$ is said to be {\em maximal} if the dimension of $\cF$ is maximal among the dimensions of facets contained in $H$.
An \emph{alcove} is the closure of any facet of maximal dimension in $\cA$. 

For example, $\cH_0$ consists of all $x \in \cA$ for which $\alpha(x) \notin \ZZ$ for all $\alpha \in \Phi$.  This is the set of points $x \in \cA$ that do not lie on  any of the hyperplanes $H_{\alpha-m}$ for any $\alpha \in \Phi$ and any $m \in \ZZ$.  Thus any connected component of $\cH_0$ is the interior of some alcove in $\cA$.  For instance, in the case of $\fsl_3$ or $\fsp_4$, these facets will be 2-dimensional.  Also for $\fsl_3$ or $\fsp_4$, any connected component of $\cH_1$ is the edge of an alcove, and any connected component of $\cH_2$ is a vertex of an alcove.


\subsubsection{Moy-Prasad filtration lattices, and generalized facets}

For each pair $(x, r)$ with $x \in \cA$ and $r \in \R$, Moy and Prasad \cite{moy-prasad:k-types} define certain $\fO$-lattices $\fg_{x, r}$ giving a filtration of $\fg$.  The parameter $r$ is referred to as the {\em depth} of the lattice.  Since we consider only the case $r=0$, we suppress $r$ from the notation throughout. 


Associated to each root $\alpha \in \Phi$ we have the root subgroup 
${\mathbf U}_\alpha$,
a $\bT(F)$-invariant, closed one-parameter subgroup of $\bG$,  and 
the root subspace 
$\fg_\alpha$, which coincides with the tangent space of ${\mathbf U}_\alpha$.  
(For the rest of this section, we reserve boldface letters for algebraic groups, and 
their non-boldface counterparts for the groups of rational points.)
Note that our groups $\bG$, $\bT$  and ${\mathbf U}_\alpha$ are in fact defined over $\Z$, and thus we can talk about the well-defined subgroups $\bG(\fO)$, ${\mathbf U}_{\alpha}(\fO)$, etc.

With this  notation (see \cite{rabinoff}*{\S 3} for more detail of the notation) 
for each $x \in \cA$, define the 
{\em parahoric subgroup} $G_x$ as 
\begin{equation*}
G_x= \langle \bT(\fO), {\mathbf U}_{\alpha}(\fP^{-\floor{\alpha(x)}})\, \mid \alpha \in \Phi \rangle;
\end{equation*}
 its 
{\em pro-unipotent radical} is 
\begin{equation*}
G_x^+ = \langle \bT(\fP +1), {\mathbf U}_{\alpha}(\fP^{1-\ceil{\alpha(x)}})\, \mid \alpha \in \Phi \rangle.
\end{equation*}
(This turns out to be equivalent to the more complicated standard definition, cf. \cite{rabinoff}*{(3.1)}). 
   
Similarly, for each $x \in \cA$, we have the corresponding lattices $\fg_x\supset \fg_x^+$ in the Lie algebra: 
\begin{equation}\label{lattice:gF}
\fg_x = \langle \fh(\fO), \fP^{-\floor{\alpha(x)}}\, X_\alpha \mid \alpha \in \Phi \rangle
\end{equation}
and 
\begin{equation}\label{lattice:gF+}
\fg_x^+ = \langle \fh(\fP), \fP^{1-\ceil{\alpha(x)}}\, X_\alpha \mid \alpha \in \Phi \rangle,
\end{equation}
where the root space $\fg_\alpha$ is spanned by the element $X_\alpha$, and $\fh=\operatorname{Lie}(\bT)$ is the Cartan subalgebra of $\fg$ corresponding to $\bT$.
(More precisely, by choosing a splitting $({\bf B}, \bT, \{x_{\alpha}\})$ of $\bG$, defined over $\Z$, we would then have the corresponding generators $X_{\alpha}$ of $\fg_\alpha$. For our classical Lie algebras, $X_{\alpha}$ are the standard generators of the corresponding root spaces; see examples in \S\ref{sec:match}  below).  

If $x, y \in \cA$ are contained in the same facet $\cF$, then we have $G_x = G_y$ and $G_x^+ = G_y^+$, as well as $\fg_x = \fg_y$ and $\fg_x^+ = \fg_y^+$. For a given facet $\cF$, we will simply write $\fg_{\cF}$ and $\fg_{\cF}^+$ for the lattices associated to any $x \in \cF$.
 We also have need of the quotient of these lattices, denoted $V_\cF = \fg_\cF \slash \fg_\cF^+$, which is a Lie algebra over $k_F$.  


In order to state DeBacker's parametrization theorem, we need to define an equivalence relation on facets, and in order to do that, we require the notion of a 
\emph{generalized  facet}. 
For each $x \in \cB(\bG)$, the set 
\begin{equation}\label{eqn:gen-facet}
\fF = \{ y \in \cB(\bG) \mid \fg_x = \fg_y \text{ and } \fg_x^+ = \fg_y^+\}
\end{equation}
is called the {\em generalized facet containing $x$}.  We say two generalized facets $\fF_1$ and $\fF_2$ are {\em strongly associate} if $A(\fF_1 \cap \cA, \cA) = A(\fF_2 \cap \cA, \cA) \neq \emptyset$, for some apartment $\cA$.  If there exists an element $g \in \bG$ such that $\fF_1$ and $g\fF_2$ are strongly associate, then we say $\fF_1$ and $\fF_2$ are {\em associate}.
For two facets $\cF_1$ and $\cF_2$ contained in a given apartment $\cA$, 
we say that $\cF_1$ and $\cF_2$ are associate if the generalized facets they determine are associate.

\begin{rem}In this paper, thanks to the explicit parametrization of orbits and Nevins' matching theorem, we 
need not interpret this notion of associate using Denef-Pas language; the fact that this notion involves the whole building and not just a single apartment  is one of the main obstructions we currently perceive to obtaining our main result for general Lie algebras. 
\end{rem}


\subsubsection{DeBacker's parametrization}

We say an element $v \in V_\cF$ is {\em degenerate} if the coset it parametrizes contains a nilpotent element ({\em i.e.} if there exists a nilpotent $X \in \fg_\cF$ such that $v = X + \fg_\cF^+$).  Let $\cI(F)$ be the set given by 
\begin{equation}\label{eqn:Fv-pairs}
\cI(F) = \{(\cF, v) \mid \cF \subset \cA \text{ is a facet, and } v \in V_\cF \text{ is a degenerate element}\}.
\end{equation}
DeBacker defines an equivalence relation $\sim$ on $\cI(F)$:
say that $(\cF_1, v_1) \sim (\cF_2, v_2)$ if and only if there exists $g \in \bG$ such that $A(\cF_1, \cA) = A(g\cF_2, \cA)$, and such that 
under the resulting natural identification of $V_{\cF_1}$ with 
$\Ad(g)V_{\cF_2}$, the elements $v_1$ and $\Ad(g)v_2$ lie in the same orbit under 
$G_x$ for any $x\in\cF_1$. 

Let $\nil(F)$ denote the set of rational nilpotent orbits in $\fg$.  Using the theory of $\fsl_2$-triples, DeBacker proves the following results regarding the relationship between the sets $\cI(F)$ and $\nil(F)$.  
\begin{lem}[DeBacker \cite{debacker:nilp}]\label{lemma:debacker-nilp}
Suppose the residue characteristic of $F$ is sufficiently large, and $(\cF, v) \in \cI(F)$.  Then 
	\begin{enumerate}
	\item (Lemma 5.3.3, $r=0$ case) There exists a unique nilpotent orbit of minimal dimension which intersects the coset $v$ nontrivially.  We denote this nilpotent orbit by $\cO(\cF,v)$.
	\medbreak
	\item (Lemma 5.4.1, $r=0$ case) The map $\gamma: \cI(F) /\sim \; \longrightarrow \nil(F)$ defined by \\$(\cF, v) \mapsto \cO(\cF, v)$ is a well-defined, surjective map.
	\end{enumerate}
\end{lem}

However, this map is not injective.
(A detailed explanation of this phenomenon is given in \cite{nevins:param}.)
To obtain a one-to-one correspondence, we must restrict to the subset of {\em distinguished pairs}.  We say a pair $(\cF,v) \in \cI(F)$ is {\em distinguished} if $v$ is not an element of any proper Levi subalgebra of the $V_\cF$.  Let 
\begin{equation}\label{eqn:IdF}
\cI^d(F) = \{ (\cF,v) \in \cI(F) \mid (\cF,v) \text{ is distinguished} \}.
\end{equation}
\begin{thm}[DeBacker \cite{debacker:nilp}, $r=0$ case of Theorem 5.6.1]
Suppose the residue characteristic of $F$ is sufficiently large.  
Then there is a bijective correspondence between $\cI^d(F)/\sim$ and the set of nilpotent orbits in $\gerg(F)$ given by the map which sends $(\cF,v)$ to $\cO(\cF,v)$.
\end{thm}

\begin{proof}
Theorem 5.6.1 from DeBacker \cite{debacker:nilp} contains this statement for a slightly different parameter space, allowing the facets $\cF$ of $\cI^d(F)$ to run over the enlarged Bruhat-Tits building of $\fg$. 
By virtue of Theorem 5.6 from Nevins \cite{nevins:param}, one may substitute the parameter space $\cI^d(F)$ defined above.
\end{proof}


\section{Explicit matching between the two parametrizations}\label{sec:match}

Following Nevins \cite{nevins:param}, we give a correspondence between the parametrization involving partitions and DeBacker's parametrization defined in terms of the building, for the cases $\fg = \fsl_n$ and $\fsp_{2n}$.  
(In order for both parametrizations to be valid, we must again assume that the residue characteristic of $F$ is not $2$, and the characteristic of $F$ itself is zero, or sufficiently large.)
We also work out the examples of $\fsl_3$ and $\fsp_4$ in complete detail. 

In 
\S \ref{sec:main} we will associate certain definable functions with each nilpotent orbit.  For that purpose it would be very convenient to use DeBacker's parametrization, however, the set $\nil(F)$ itself is more easily understood through Waldspurger's parametrization via partitions. Thus, 
it is necessary to understand the explicit matching between these two parametrizations.


\subsection{The matching for $\fsl_n$}\label{sec:match-sln}

For each partition $\lambda = (\lambda_1, \lambda_2, \ldots, \lambda_t)$ of $n$, and each diagonal matrix $D = \diag(d_1, d_2, \ldots, d_n) \in \bT$, define the set 
$$I_\lambda = \{1, 2, \ldots, n\} \setminus \{\lambda_1, \lambda_1 + \lambda_2, \ldots, \sum_i \lambda_i = n\}.$$
$I_\lambda$ represents the set of locations of the nonzero entries of the matrices $J_\lambda D(d) \in \cO_\lambda(F)$  described in Proposition \ref{prop:partitions}.  For each $i \in I_\lambda$, the value $d_{i+1}$ is the $(i, i+1)$-entry of $X$, and all remaining entries are zero. 

Recall the hyperplanes $H_{\varphi+n} = \{x \in \cA \mid \varphi(x)= -n\}$, defined for roots $\varphi \in \Phi$.  Also recall the standard notation $\alpha_i = e_i - e_{i+1}$ for the simple roots of $\SL_n$.
With the notation as above, define 
$$H_{\lambda,D} = \bigcap_{i \in I_\lambda} H_{\alpha_i + \val(d_{i+1})} \subseteq \cA.$$  Note that when $\lambda = (1, 1, \ldots, 1)$, we have $X = 0$ and $I_\lambda = \emptyset$.  Nevins \cite{nevins:param} states that the zero orbit corresponds to the associate class of the interior of any alcove in the apartment.  The following theorem of Nevins establishes the remainder of the correspondence between the two parametrizations of the nilpotent orbits for $\fsl_n$.  

\begin{thm}[\cite{nevins:param}, Theorem 2 with $r=0$]\label{thm:sln-corresp}
Let $\lambda, D$, and $H_{\lambda, D} \subset \cA$ be as above, and let $\cF$ be any facet of maximal dimension in $H_{\lambda, D}$.  For any $x \in \cF$, we have $X = J_\lambda D \in \fg_\cF$; set $v$ to be its image in $V_\cF$.  Then 
$(\cF,v)\in \cI^d(F)$ and $\cO(\cF,v) = \Ad(\fsl_n(F))\, X$.
\end{thm}


\subsection{The matching for $\fsp_{2n}$}\label{sec:match-sp2n}

Let $X\in \cO_\lambda(F)$ be the nilpotent element 
corresponding to the $n$-tuple $\cQbar$, as in Theorem \ref{reps}. 
For each odd $j$, let $$I_j = \{1,\ldots,\tfrac{1}{2}jm_j\} \setminus
\{j,2j,\ldots,\tfrac{1}{2}jm_j\}$$ and let $S_j = S_j^1$ denote the set
of simple roots
$$S_j^1 = \{e_{s_j+k}-e_{s_j+k+1} \, | \, k\in I_j\}.$$
For each even $j$, suppose $Q_j=q_0^m\oplus Q\an$ is the minimal matrix representative for $\cQ_j$, where $m$ is the Witt index of $Q_j$ ($0\leq 2m \leq m_j$), and set $M_j = (\frac{1}{2}j-1)m_j$. 
Then we take $S_j = S_j^1\cup S_j^2$, where
$$ S_j^1 = \{e_{s_j+k}-e_{s_j+k+m_j}\,|\, 1\leq k \leq M_j \} \cup
\{e_{s_j+M_j+2i-1} + e_{s_j+M_j+2i}\,|\, 1\leq i \leq m \}$$ 
and 
$$ S_j^2 = \{2e_{s_j+M_j+i}\,|\, 2m<i\leq m_j\}.$$ 
If $Q\an = \diag(a_{2m+1},\ldots,a_{m_j})$, define for each root $\alpha_i = 2e_{s_j+M_j+i}$ the integer $v_{\alpha_i} = \val(a_i)$ for $2m+1\leq i \leq m_j$. 
Let $H_{\lambda,\cQbar}$ be the common intersection (over all $j$) of the hyperplanes $H_\alpha$ for $\alpha\in S_j^1$ and $H_{\alpha+v_\alpha}$ for $\alpha\in S_{2j}^2$. 
Finally, the following theorem of Nevins gives the correspondence between the two parametrizations of nilpotent orbits for $\fsp_{2n}$:
\begin{thm}[\cite{nevins:param}, Theorem 4 with $r=0$]\label{thm:sp2n-corresp}
The affine subspace $H_{\lambda,\cQbar} \subset \cA$ is a nonempty union of facets. 
Let $\cF$ be any maximal facet in $H_{\lambda,\cQbar}$, and let $v$ denote the projection of $X$ in $V_{\cF}$. Then $(\cF,v)\in \cI^d(F)$ and $\cO(\cF,v) = \Ad(\fsp_{2n}(F))\, X$.
\end{thm}


\subsection{Example:  The correspondence, in the case of $\fsl_3$}\label{example:sl3}

We examine these parametrizations and their correspondence in the case of the Lie algebra $\fg=\fsl_3$. Representatives $X_\lambda$ for the nilpotent orbits $\cO_\lambda$ are given above in Example \ref{example:sl3_partitions}.  Following the construction given in \S \ref{sec:match-sln}, we compute the sets $H_{\lambda, D(d)}$ as in Theorem \ref{thm:sln-corresp}, where $D = D(d) = \diag(1, 1, \ldots, 1, d)$, for $d \in F^\times$.   We then determine the maximal facet $\cF \subseteq H_{\lambda,D}$ corresponding to each orbit $\cO_{(1,1,1)}$ and $\cO_{(2,1)}$, and to each of the $F$-rational nilpotent orbits contained in $\cO_{(3)}$.  

For the partition $\lambda = (1,1,1),$ we have $I_{(3)} =\emptyset$ and $X = 0$.  In this trivial case, the corresponding maximal facet is the interior of any alcove in the apartment $\cA$.  The standard apartment for $\fsl_3$ is shown below in Figure 1.  We may choose the alcove given by the outlined region, which is bounded by the hyperplanes $H_{\alpha_1}$, $H_{\alpha_2}$, and $H_{(\alpha_1 + \alpha_2) - 1}$.  This is denoted  by $\cF_1$ below in Figure 2.

When $\lambda = (2,1),$ we have $I_{(2,1)} = \{1\}$ and $H_{(2,1), D(d)} = H_{\alpha_1}$ for any $d \in F^\times$.  Any facet of maximal dimension in $H_{(2,1),D}$, is therefore an edge in an alcove of $\cA$.  Since the three edges of any alcove are associates, it suffices to consider a single edge, 
denoted $\cF_2$.

For $\lambda = (3),$ we have $I_{(3)} = \{1, 2\}$ and $H_{(3), D(d)} = H_{\alpha_1} \cap H_{\alpha_2 + \val(d)}.$  Recall that we have $X_{(3)} = \left(\begin{smallmatrix}
	0 & 1 & 0\\
	0 & 0 & 1\\
	0 & 0 & 0
	\end{smallmatrix} \right)$ and this orbit splits into $3\cdot\gcd(3, q-1)$ orbits $\cO_\lambda(F)$ whose representatives are given by
\begin{equation}\label{eqn:Xd}
X_d = \begin{pmatrix}
	0 & 1 & 0\\
	0 & 0 & d\\
	0 & 0 & 0
	\end{pmatrix},\,\, \text{one for each distinct equivalence class of $d$ in $F^\times/(F^\times)^3$.}
\end{equation}
   When $3 \nmid (q-1)$ we have $d \in \{1, \varpi, \varpi^2\}$.  When $d \mid (q-1)$, we fix a non-cubic unit $\varepsilon \in F^\times$, and have $d \in \{1, \varepsilon, \varepsilon^2, \varpi, \varepsilon \varpi, \varepsilon^2 \varpi, \varpi^2, \varepsilon \varpi^2,  \varepsilon^2 \varpi^2\}$.   In any case, $H_{(3), D(d)}$ will be a single point, thus any facet of maximal dimension in $H_{(3),D(d)}$ will consist of a single element.  
   
Specifically, when $\val(d)=0$, then $H_{(3), D(d)}=\{0\}$, with corresponding facet denoted by $\cF_3$ in Figure 2.    Taking $\val(d) = 1$ gives $H_{(3), D(d)} = H_{\alpha_1} \cap H_{\alpha_2+1}$ which is not in the chosen alcove.  For our purposes in handling definability, it is convenient to fix a single alcove.  With this in mind, we note that $\bG(F)$ acts on $\cA$ via the affine Weyl group; 
and so, reflecting this point across the hyperplane $H_{\alpha_1 + \alpha_2}$ to the upper-right vertex of the alcove, we see that this facet is associate with 
the facet denoted by $\cF_4$.  Similarly, $\val(d) = 2$ gives $H_{(3), D(d)} = H_{\alpha_1} \cap H_{\alpha_2+2}$, which maps to facet $\cF_5$ under the affine Weyl group action ({\em e.g.} by reflecting across $H_{\alpha_2}$ and then $H_{\alpha_1+1}$). 



\begin{figure}[htbp]
\centering
\begin{tikzpicture}[scale=(1.2),dot/.style={fill=black,circle,inner
        sep=1.5pt}]
\begin{scope}
\draw [thick](-4,-3.4641) edge (-4,3.4641);  
\draw [thick](-4,-3.4641) edge (4,-3.4641);  
\draw [thick](-4,3.4641) edge (4,3.4641);  
\draw [thick](4,-3.4641) edge (4,3.4641);  


\node[fill=blue,circle,inner sep=2pt] at (0,0) {};

\node[fill=blue,circle,inner sep=2pt] at (1,{-1.732}) {};

\node[fill=blue,circle,inner sep=2pt] at (2,-3.4641) {};

\draw [dotted](-4,0) edge (-2,3.4641);
\draw [dotted](-4,-3.4641) -- (0,3.4641) node [sloped, pos = .2,
  above] {$H_{\alpha_2-1}$};
\draw [thick,dotted,blue](-2,-3.4641) -- (2,3.4641) node [sloped, pos = .1,
  above] {$H_{\alpha_2}$};  
\draw [dotted](0,-3.4641) -- (4,3.4641) node [sloped, pos = .1,
  above] {$H_{\alpha_2+1}$};
\draw [dotted](2,-3.4641) edge (4,0);

\draw [dotted](4,0) edge (2,3.4641); 
\draw [dotted](4,-3.4641) -- (0,3.4641) node [sloped, pos = .9,
  above] {$H_{\alpha_1-1}$};
\draw [thick,dotted,blue](2,-3.4641) -- (-2,3.4641) node [sloped, pos = .9,
  above] {$H_{\alpha_1}$};  
\draw [dotted](0,-3.4641) -- (-4,3.4641) ;
;

\draw [dotted](-4,0) edge (-2,-3.4641);

\draw [dotted](-4,{-1.732}) edge (4,{-1.732});
\draw [thick,dotted,blue](-4,0) -- (4,0) node [sloped, pos = .9,
  above] {$H_{\alpha_1+\alpha_2}$};
\draw [dotted](-4,{1.732}) -- (4,{1.732}) node [sloped, pos = .8,
  above] {$H_{(\alpha_1+\alpha_2)-1}$};

\draw[thick,blue] (0,0) edge (-1, {1.732});
\draw[thick,blue] (0,0) edge (1, {1.732});
\draw[thick,blue] (-1,{1.732}) edge (1, {1.732});

\end{scope}
\end{tikzpicture}
\caption{Standard affine apartment of $\fsl_3(F)$.  Blue edges outline an alcove, and dotted lines indicate affine hyperplanes.  Blue dots indicate the sets  $H_{(3),D(d)}$.}
\label{picture}
\end{figure}
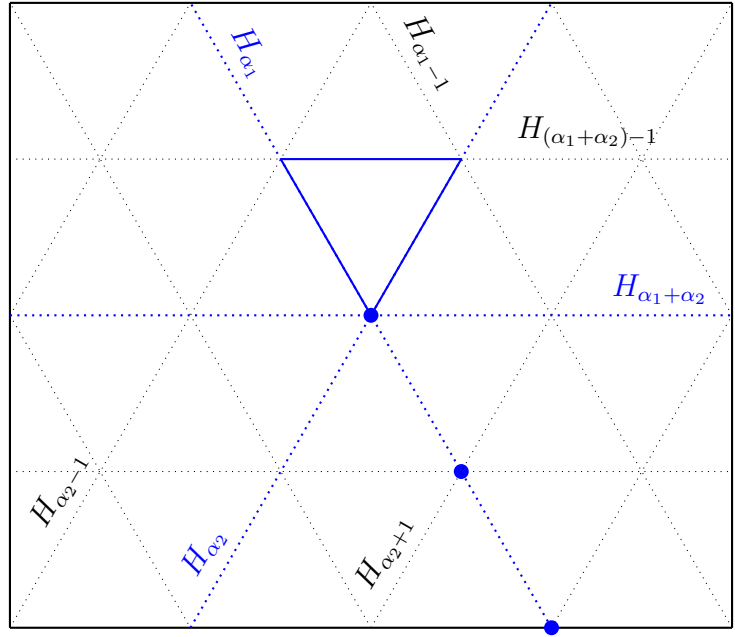

\begin{figure}[htbp]
\centering
\begin{tikzpicture}[scale=(1.75),dot/.style={fill=black,circle,inner
        sep=1.5pt}]
\begin{scope}

\node[fill=black,circle,inner sep=1pt] at (-1,{1.732}) {} node[left] at (-1,{1.732}) {$\cF_5$}; 
\node[fill=black,circle,inner sep=1pt] at (1,{1.732}) {} node[right] at (1,{1.732}){$\cF_4$};
\node[fill=black,circle,inner sep=1pt] at (0,0) {}  node[below] {$0 = \cF_3$};

\node at (0,1.03923) {$\cF_1$};

\draw[thick,blue] (0,0) edge (-1, {1.732});
\node[blue, sloped, left] at (-.5, 0.866) {$H_{\alpha_1}$}; 
\draw[thick,blue] (0,0) edge (1, {1.732});
\node[sloped, blue, right] at (1/2, 0.866) {$H_{\alpha_2}$}; 
\node at (0.55, 0.65) {$\cF_2$}; 
 
\draw[thick,blue] (-1,{1.732}) edge (1, {1.732});
\node[sloped, blue, above] at (0, 1.81865) {$H_{(\alpha_1+\alpha_2) - 1}$}; ;

\end{scope}
\end{tikzpicture}
\caption{Facets $\cF_i$ in the given alcove of the standard affine apartment of $\fsl_3(F)$.  The facet $\cF_2$ is associate to the other edges.}
\label{picture}
\end{figure}
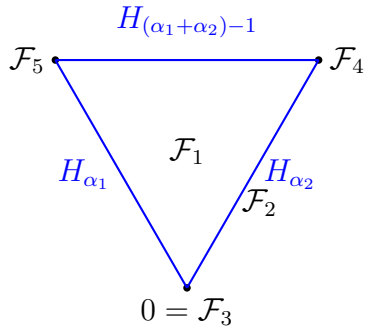

In order to calculate the lattices associated to each facet $\cF_i$, we first determine the root spaces $\gerg_\alpha$ for $\alpha\in\Phi$. By \eqref{eqn:sln-roots} (with $\alpha_i = e_i - e_{i+1}$), we have
$$\Phi = \{\, \alpha_1,\, \alpha_2,\, \alpha_1+\alpha_2,\, -\alpha_1,\, -\alpha_2,\, -(\alpha_1+\alpha_2) \,\}.$$

Each root space is one-dimensional, and the generators for the six root spaces are given by the matrices
\begin{center}
\begin{tabular}{lll}
$X_{\alpha_1} = \begin{pmatrix}
	0 & 1 & 0\\
	0 & 0 & 0\\
	0 & 0 & 0
	\end{pmatrix}$ & 
$X_{\alpha_2} = \begin{pmatrix}
	0 & 0 & 0\\
	0 & 0 & 1\\
	0 & 0 & 0
	\end{pmatrix}$ & 
$X_{\alpha_1+\alpha_2} = \begin{pmatrix}
	0 & 0 & 1\\
	0 & 0 & 0\\
	0 & 0 & 0
	\end{pmatrix}$\\
& & \\

$X_{-\alpha_1} = \begin{pmatrix}
	0 & 0 & 0\\
	1 & 0 & 0\\
	0 & 0 & 0
	\end{pmatrix}$ & 
$X_{-\alpha_2} = \begin{pmatrix}
	0 & 0 & 0\\
	0 & 0 & 0\\
	0 & 1 & 0
	\end{pmatrix}$ & 
$X_{-(\alpha_1+\alpha_2)} = \begin{pmatrix}
	0 & 0 & 0\\
	0 & 0 & 0\\
	1 & 0 & 0
	\end{pmatrix}$\\
\end{tabular}
\end{center}

Sample calculations for the lattices associated to the facets $\cF_1$ and $\cF_4$ are given in detail below, followed by a table with full results for each of the five facets mentioned above.  Although the lattices associated to $H_{(3), D(d)}$ when $\val(d) = 1$ will be different from the lattices $\fg_{\cF_4}$ and $\fg_{\cF_4}^+$, the quotients will be identical.  Therefore we may consider $\cF_4$ when determining the image $v$ of $X_d$ in the quotient.  The case $\val(d)=2$ is handled similarly. 
In the same spirit, we compute $\fg_{\cF}$ for the edge that is labelled $\cF_2$ in Figure 2; though the lattices $\fg_{\cF}$ and $\fg_{\cF}^+$ are different for the other two associate edges, the quotients $V_{\cF}$ for them are isomorphic.  

\begin{enumerate}
\item [{\bf Facet $\cF_1$:\, }] 
$x \in \cF_1$ if and only if $0 < \alpha_1(x), \alpha_2(x), (\alpha_1+\alpha_2)(x) < 1.$
For any $x \in \cF_1$ this gives $\floor{\alpha_1(x)} = 0$ and $\ceil{\alpha_1(x)} = 1$, while $\floor{-\alpha_1(x)} = -1$ and $\ceil{-\alpha_1(x)} = 0$.  We have
$\fP^{-\floor{\alpha_1(x)}} = \fP^0 = \fO$ and $\fP^{-\floor{-\alpha_1(x)}} = \fP^1 = \fP$, hence
$$\fP^{-\floor{\alpha_1(x)}} \, X_{\alpha_1} 
 = \begin{pmatrix}
	0 & \fO & 0\\
	0 & 0 & 0\\
	0 & 0 & 0
	\end{pmatrix} \text{ and } 	
\, \fP^{-\floor{-\alpha_1(x)}}\, X_{-\alpha_1}  = \begin{pmatrix}
	0 & 0 & 0\\
	\fP & 0 & 0\\
	0 & 0 & 0
	\end{pmatrix}.$$
We may similarly calculate $\fP^{-\floor{\varphi(x)}} \, X_\phi$ for each of the roots $\varphi = \pm \alpha_2,\,  \pm(\alpha_1+\alpha_2)$.   Finally, $\fh(\fO) = \begin{pmatrix}
	\fO & 0& 0 \\
	0 & \fO & 0\\
	0 & 0 & \fO
	\end{pmatrix}$, and so we obtain
$\fg_{\cF_1} = \begin{pmatrix}
	\fO & \fO& \fO \\
	\fP & \fO & \fO\\
	\fP & \fP & \fO
	\end{pmatrix}.$
	
Next, $\fP^{1-\ceil{\alpha_1(x)}} = \fO$ and $\fP^{1-\ceil{-\alpha_1(x)}} = \fP$, and similarly for $\pm \alpha_2$ and $\pm (\alpha_1 + \alpha_2)$.  Computing the corresponding lattice representative for each root, and using the fact that $\fh(\fP) = \begin{pmatrix}
	\fP & 0& 0 \\
	0 & \fP & 0\\
	0 & 0 & \fP
	\end{pmatrix},$ 
we find that  
$\fg_{\cF_1}^+ = \begin{pmatrix}
	\fP & \fO& \fO \\
	\fP & \fP & \fO\\
	\fP & \fP & \fP
	\end{pmatrix}.$
Since $k_F = \fO / \fP$, taking the quotient gives 
$V_{\cF_1} = \fg_{\cF_1} \slash \fg_{\cF_1}^+ 	=\begin{pmatrix}
	k_F & 0 & 0\\
	0 & k_F & 0\\
	0 & 0 & k_F
	\end{pmatrix}.$

\item [{\bf Facet $\cF_4$\,}]:  This vertex lies on the hyperplanes $H_{\alpha_1-1}$, $H_{\alpha_2}$, and $H_{(\alpha_1 + \alpha_2)-1}$, so $x \in \cF_4$ if and only if $\alpha_1(x) = (\alpha_1 + \alpha_2)(x) = 1$ and $\alpha_2(x)=0$.  Thus  
$\fP^{-\floor{\varphi(x)}} = \fP^{-1}$ and $\fP^{1-\ceil{\varphi(x)}} = \fO$ for $\varphi = \alpha_1$ and $\alpha_1 + \alpha_2$, and similarly $\fP^{-\floor{\varphi(x)}} = \fP$ and $\fP^{1-\ceil{\varphi(x)}} = \fP^2$ for $\varphi = -\alpha_1$ and 
$-(\alpha_1 + \alpha_2)$.  For the remaining roots, we have $\fP^{-\floor{\pm \alpha_2(x)}} = \fO$ and $\fP^{1-\ceil{\pm \alpha_2(x)}} = \fP$.  Adding the corresponding matrix representatives, we obtain
$$\fg_{\cF_4} = \begin{pmatrix}
	\fO & \fP^{-1}& \fP^{-1} \\
	\fP & \fO & \fO\\
	\fP & \fO & \fO
	\end{pmatrix} \;\; \text{and} \;\;
\fg_{\cF_4}^+ = \begin{pmatrix}
	\fP & \fO & \fO \\
	\fP^2 & \fP & \fP\\
	\fP^2 & \fP & \fP
	\end{pmatrix}.$$
Identifying $\gP^a$ with $\varpi^a \gO$, we get isomorphisms of $\gP^a
/ \gP^{a+1}$ with $k_F=\gO/\gP$. Thus taking the quotient gives 
$V_{\cF_4} = \begin{pmatrix}
	k_F & k_F  & k_F\\
	k_F & k_F & k_F\\
	k_F & k_F & k_F
	\end{pmatrix}.$
\bigbreak
\end{enumerate}
\medbreak

\begin{center}
\begin{tabular}{|c|c|c|}
\hline
Parahoric $\fg_\cF$ & Pro-unipotent $\fg_\cF^+$ & $V_\cF$\\
\hline
$\fg_{\cF_1} = \begin{pmatrix}
	\fO & \fO& \fO \\
	\fP & \fO & \fO\\
	\fP & \fP & \fO
	\end{pmatrix}$
& $\fg_{\cF_1}^+ = \begin{pmatrix}
	\fP & \fO& \fO \\
	\fP & \fP & \fO\\
	\fP & \fP & \fP
	\end{pmatrix}$
& $V_{\cF_1} =\begin{pmatrix}
	k_F & 0 & 0\\
	0 & k_F & 0\\
	0 & 0 & k_F
	\end{pmatrix}$\\
\hline
$\fg_{\cF_2} = \begin{pmatrix}
	\fO & \fO& \fO \\
	\fP & \fO & \fO\\
	\fP & \fO & \fO
	\end{pmatrix}$
& $\fg_{\cF_2}^+ = \begin{pmatrix}
	\fP & \fO & \fO \\
	\fP & \fP & \fP\\
	\fP & \fP & \fP
	\end{pmatrix}$
& $V_{\cF_2} = \begin{pmatrix}
	k_F & 0 & 0\\
	0 & k_F & k_F\\
	0 & k_F & k_F
	\end{pmatrix}$\\	
\hline
$\fg_{\cF_3} = \begin{pmatrix}
	\fO & \fO& \fO \\
	\fO & \fO & \fO\\
	\fO & \fO & \fO
	\end{pmatrix}$
& $\fg_{\cF_3}^+ = \begin{pmatrix}
	\fP & \fP & \fP \\
	\fP & \fP & \fP\\
	\fP & \fP & \fP
	\end{pmatrix}$
& $V_{\cF_3} = \begin{pmatrix}
	k_F & k_F & k_F\\
	k_F & k_F & k_F\\
	k_F & k_F & k_F
	\end{pmatrix}$\\
\hline
$\fg_{\cF_4} = \begin{pmatrix}
	\fO & \fP^{-1}& \fP^{-1} \\
	\fP & \fO & \fO\\
	\fP & \fO & \fO
	\end{pmatrix}$
& $\fg_{\cF_4}^+ = \begin{pmatrix}
	\fP & \fO & \fO \\
	\fP^2 & \fP & \fP\\
	\fP^2 & \fP & \fP
	\end{pmatrix}$
& $V_{\cF_4} = \begin{pmatrix}
	k_F & k_F & k_F\\
	k_F & k_F & k_F\\
	k_F & k_F & k_F
	\end{pmatrix}$\\
\hline
$\fg_{\cF_5} = \begin{pmatrix}
	\fO & \fO& \fP^{-1} \\
	\fO & \fO & \fP^{-1}\\
	\fP & \fP & \fO
	\end{pmatrix}$
& $\fg_{\cF_5}^+ = \begin{pmatrix}
	\fP & \fP & \fO \\
	\fP & \fP & \fO\\
	\fP^2 & \fP^2 & \fP
	\end{pmatrix}$
& $V_{\cF_5} = \begin{pmatrix}
	k_F & k_F & k_F\\
	k_F & k_F & k_F\\
	k_F & k_F & k_F
	\end{pmatrix}$\\
\hline
\end{tabular}
\end{center}

Finally, we determine the image $v$ of $X$ in $V_\cF$ for each representative $X_\lambda$ and $X_d$ as above.  
The nilpotent orbit $\cO_{(1,1,1)}$ has representative $X_{(1,1,1)} = \mathbf{0}$, with corresponding facet $\cF_1$.  Its image $v$ in $V_{\cF_1}$ is simply the zero matrix.   Similarly, $\cO_{(2,1)}$ has representative 
$X_{(2,1)} = \begin{pmatrix}
	0 & 1 & 0\\
	0 & 0 & 0\\
	0 & 0 & 0
	\end{pmatrix}$, whose corresponding facet is associate to $\cF_2$.  Its image in $V_{\cF_2}$ is 
$v_{(2,1)} = \begin{pmatrix}
	0 & 1 & 0\\
	0 & 0 & 0\\
	0 & 0 & 0
	\end{pmatrix}.$
For $\lambda=(3)$, with $d= \varepsilon^a \varpi^b$ ($0 \leq a, b \leq 2$) and $X_d$ as in Equation \ref{eqn:Xd}, 
the image of $X_d$ in the quotient $V_{F_{b+2}}$ is $v_d = \begin{pmatrix}
	0 & 1 & 0\\
	0 & 0 & \ac(d) \\
	0 & 0 & 0
	\end{pmatrix}.$

\subsection{Example: $\mathfrak{sp_4}$}\label{example:sp4}

We now examine the two parametrizations and their correspondence in the case of the Lie algebra $\gerg=\mathfrak{sp_4}$.  
As in Theorem \ref{thm:sp4-partitions}, consider only the partitions of $4$ whose odd parts have even multiplicity, {\em i.e.} $\lambda = (4), \, (2, 2), \, (2, 1, 1),$ and $(1,1,1,1).$  Each orbit $\cO_\lambda$ splits into a certain number of $F$-rational orbits, depending on $\lambda$. (For details, see \cite{nevins:param}, Table 1.)   Below we give the details for the rational nilpotent orbits contained in the
algebraic orbit $\cO_{(4)}$.  

The partition $\lambda = (4)$ corresponds to the nilpotent matrix 
$$X = \begin{pmatrix}
0 & 1 & 0 & 0 \\
0 & 0 & 0 & 1 \\
0 & 0 & 0 & 0 \\
0 & 0 & -1 & 0\end{pmatrix}
= Y J_{(4)}\,Y^{-1},\, \text{ where } 
Y = \begin{pmatrix}
1 & 1 & 1 & 0 \\
0 & 1 & 1 & 1 \\
0 & 0 & 0 & -1\\
0 & 0 & 1 & 1\end{pmatrix}.$$

We have $m_4(\lambda) = 1$ and $m_i(\lambda) = 0$ for $i=1, 2, 3$, and therefore the vector space
$V\simeq F^4$ satisfies $V = V(4)$.  
By Theorem \ref{reps}, we may use equation (\ref{jEven}) with $X = X\mid_{V(4)}$, $j=4$, $N = 2$, and $m_j = 1$ to determine representatives of the $F$-rational nilpotent orbits in $\cO_\lambda(F)$. 
By the results of Section \ref{subsec:sp_2n-nilp}, there are four minimal matrix representatives of quadratic forms of dimension $m_4=1,$
given by 
$$X_a = \begin{pmatrix}
0 & 1 & 0 & 0 \\
0 & 0 & 0 & a \\
0 & 0 & 0 & 0 \\
0 & 0 & -1 & 0\end{pmatrix},$$
such that $a$ runs over the set $\{1,\, \ve,\, \varpi,\, \ve\varpi\}$, where $\varepsilon$ is a fixed non-square unit in $F$.

We now turn to the correspondence in Section \ref{sec:match-sp2n}. 
Fixing $X = X_a$, we have $S_j = \emptyset$ for $j\neq 4$. For $j=4$, by the given construction we see that the Witt index $m$ of $Q_4$ is equal to $0$,  and $M_4 = 1$. 
Thus $$S_4^1 = \{e_1-e_2\}, \quad S_4^2 = \{2e_2\}.$$
We have $Q\an = \diag(a)$, so for $\alpha_1 = 2e_{M_4+1}= 2e_2$, we have $v_{\alpha_1} = \val(a)$.  Thus
$$H_{\lambda,\cQbar} = H_{e_1-e_2} \cap H_{2e_2+\val(a)}.$$  
Now $\val(a)$ is either $0$ or $1$, since $\varepsilon$ is a unit.  
Figure \ref{picture} shows the standard apartment of $\mathfrak{sp}_4$, along with the hyperplanes 
$H_{e_1-e_2}$, $H_{2e_2}$, and $H_{2e_2+1}$ and the associated intersections $H_{\lambda,\cQbar}$ of these
hyperplanes. 
From the diagram, it is clear that there is a unique maximal facet $\cF_a$ (vertex) in each set $H_{\lambda,\cQbar}$, and $\cF_a$ consists of a single element.


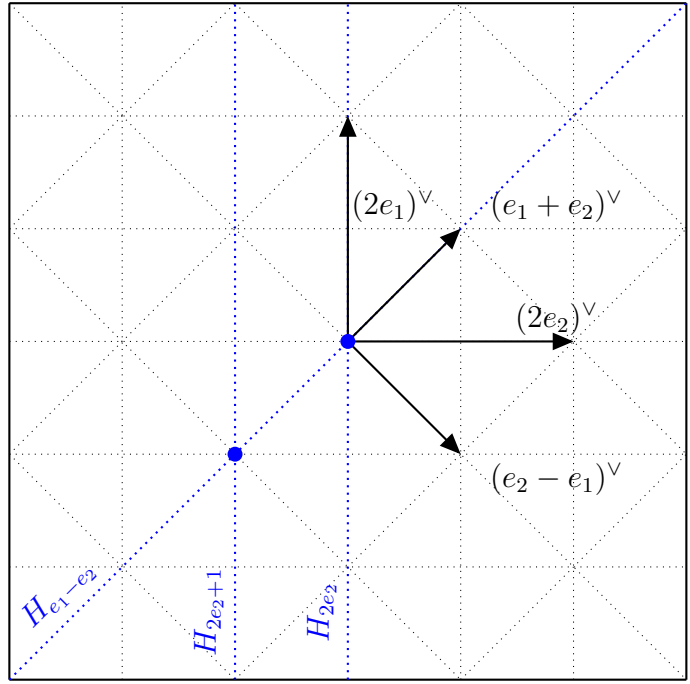
\begin{figure}[htbp]
\centering
\begin{tikzpicture}[scale=(1.5),dot/.style={fill=black,circle,inner
        sep=1.5pt}]
\begin{scope}
\draw [thick](0,0) edge (0,6);
\draw [thick](0,0) edge (6,0);
\draw [thick](0,6) edge (6,6);
\draw [thick](6,0) edge (6,6);
\draw [dotted](0,1) edge (6,1);
\draw [thick,dotted,blue](2,0) -- (2,6) node [sloped, pos = .1,
  above] {$H_{2e_2+1}$};
\draw [thick,dotted,blue](3,0) -- (3,6) node [sloped, pos = .1,
  above] {$H_{2e_2}$};
\draw [dotted](0,4) edge (6,4);
\draw [dotted](0,5) edge (6,5);
\draw [dotted](1,0) edge (1,6);
\draw [dotted](0,2) edge (6,2);
\draw [dotted](0,3) edge (6,3);
\draw [dotted](4,0) edge (4,6);
\draw [dotted](5,0) edge (5,6);

\draw [dotted](0,2) edge (4,6);
\draw [dotted](0,4) edge (2,6);
\draw [thick,dotted,blue](0,0) -- (6,6) node [sloped, pos = .1, above]
      {$H_{e_1-e_2}$};
\draw [dotted](2,0) edge (6,4);
\draw [dotted](4,0) edge (6,2);

\draw [dotted](0,2) edge (2,0);
\draw [dotted](0,4) edge (4,0);
\draw [dotted](0,6) edge (6,0);
\draw [dotted](2,6) edge (6,2);
\draw [dotted](4,6) edge (6,4);


\draw [thick,-triangle 45](3,3) -- (3,5) node at (3.4,4.2)
      {$(2e_1)^\vee$};
\draw [thick,-triangle 45](3,3) -- (4,2) node at (4.85,1.78)
      {$(e_2-e_1)^\vee$};
\draw [thick,-triangle 45](3,3) -- (5,3) node at (4.85,3.2)
      {$(2e_2)^\vee$};
\draw [thick,-triangle 45](3,3) -- (4,4) node at (4.85,4.2) {$(e_1+e_2)^\vee$};

\node[fill=blue,circle,inner
        sep=2pt] at (3,3) {};
\node[fill=blue,circle,inner
        sep=2pt] at (2,2) {};
\end{scope}
\end{tikzpicture}
\caption{The standard affine apartment of
  $\sp_4(F)$. Arrows indicate positive co-roots, and dotted lines indicate
  affine hyperplanes. 
 Blue dots indicate the sets $H_{\lambda,\cQbar}$.}
\label{picture}
\end{figure}

In order to calculate the associated lattices, we first determine the root spaces $\gerg_\alpha$ for $\alpha\in\Phi$. 
By \eqref{eqn:sp2n-roots}, we have
$$\Phi = \{ \pm(e_1 - e_2),\, \pm(e_1 + e_2),\, \pm 2e_1,\, \pm 2e_2 \}.$$
Each root space is one-dimensional, and the generators for the root spaces are given by the matrices
\smallbreak
\begin{center}
\begin{tabular}{ll}
$X_{e_1 - e_2} = \begin{pmatrix}
0 & 1 & 0 & 0 \\
0 & 0 & 0 & 0 \\
0 & 0 & 0 & 0 \\
0 & 0 & -1 & 0
\end{pmatrix}$
&
$X_{e_2 - e_1} = \begin{pmatrix}
0 & 0 & 0 & 0 \\
1 & 0 & 0 & 0 \\
0 & 0 & 0 & -1 \\
0 & 0 & 0 & 0
\end{pmatrix}$\\
\\

$X_{e_1 + e_2} = \begin{pmatrix}
0 & 0 & 0 & 1 \\
0 & 0 & 1 & 0 \\
0 & 0 & 0 & 0 \\
0 & 0 & 0 & 0
\end{pmatrix}$
&
$X_{-(e_1 + e_2)} = \begin{pmatrix}
0 & 0 & 0 & 0 \\
0 & 0 & 0 & 0 \\
0 & 1 & 0 & 0 \\
1 & 0 & 0 & 0
\end{pmatrix}$\\
\\

$X_{2e_1} = \begin{pmatrix}
0 & 0 & 1 & 0 \\
0 & 0 & 0 & 0 \\
0 & 0 & 0 & 0 \\
0 & 0 & 0 & 0
\end{pmatrix}$ 
&
$X_{-2e_1} = \begin{pmatrix}
0 & 0 & 0 & 0 \\
0 & 0 & 0 & 0 \\
1 & 0 & 0 & 0 \\
0 & 0 & 0 & 0
\end{pmatrix}$ \\
\\

$X_{2e_2} = \begin{pmatrix}
0 & 0 & 0 & 0 \\
0 & 0 & 0 & 1 \\
0 & 0 & 0 & 0 \\
0 & 0 & 0 & 0
\end{pmatrix}$ 
&
$X_{-2e_2} = \begin{pmatrix}
0 & 0 & 0 & 0 \\
0 & 0 & 0 & 0 \\
0 & 0 & 0 & 0 \\
0 & 1 & 0 & 0
\end{pmatrix}$ \\
\end{tabular}
\end{center}

Identify $\cF_a$ with the single element that it contains. By
referencing Figure \ref{picture}, we can calculate $\alpha(\cF_a)$
for $\alpha\in \Phi$. Clearly $\alpha(\cF_a)=0$ for all $\alpha$
when $\cF_a=0$ (\emph{i.e.} when $a=1,\ve$). Suppose $a = \varpi$ or
$\ve\varpi$, so that $\cF_a=\frac{-1}{2}(e_1+e_2)^\vee$. This gives the values $\pm2e_1(\cF_a) = \mp1$, \, $\pm2e_2(\cF_a) = \mp1$, \, $\pm(e_1+e_2)(\cF_a) = \mp1$, and $\pm(e_1-e_2)(\cF_a) = 0.$

Using the definition of the lattices $\gerg_\cF$ and
$\gerg_\cF^+$ given in \eqref{lattice:gF} and \eqref{lattice:gF+} respectively, we
compute the following:\\If $a = 1$ or $\ve$, we have
$$
\gerg_{\cF_a} = \begin{pmatrix}
\gO & \gO & \gO & \gO \\
\gO & \gO & \gO & \gO \\
\gO & \gO & \gO & \gO \\
\gO & \gO & \gO & \gO
\end{pmatrix} \,\, \text{ and } \,\,
\gerg_{\cF_a}^+ = \begin{pmatrix}
\gP & \gP & \gP & \gP \\
\gP & \gP & \gP & \gP \\
\gP & \gP & \gP & \gP \\
\gP & \gP & \gP & \gP
\end{pmatrix}.$$
If $a = \varpi$ or $\ve\varpi$, we have
$$
\gerg_{\cF_a} = \begin{pmatrix}
\gO & \gO & \gP & \gP \\
\gO & \gO & \gP & \gP \\
\gP\inv & \gP\inv & \gO & \gO \\
\gP\inv & \gP\inv & \gO & \gO
\end{pmatrix} \,\, \text{ and } \,\,
\gerg_{\cF_a}^+= \begin{pmatrix}
\gP & \gP & \gP^2 & \gP^2 \\
\gP & \gP & \gP^2 & \gP^2 \\
\gO & \gO & \gP & \gP \\
\gO & \gO & \gP & \gP
\end{pmatrix}.$$

It is clear that we have $k_F$ in each entry of the quotient
in both cases, hence $V_\cF \simeq \sp_4(k_F)$. Finally, we determine the image $v_a$ of $X_a$ in
$V_\cF$. 
Then we have
$$v_a =
\begin{pmatrix}
0 & 1 & 0 & 0 \\
0 & 0 & 0 & \ac(a)  \\
0 & 0 & 0 & 0 \\
0 & 0 & -1  & 0
\end{pmatrix} \,\, \text{ for each $a \in \{1, \varepsilon, \varpi, \varepsilon \varpi\}$.}$$


\section{Shalika germs}\label{sec:main}




\subsection{The main results}
Here we prove that the so-called \emph{provisional} Shalika germs are motivic
(in the terminology of \cite{kottwitz:clay}*{\S 6}). 
Harish-Chandra defined Shalika germs on the full Lie algebra, using their homogeneity (see \cite{kottwitz:clay}*{\S 17} for a detailed discussion). 
Here we will show, roughly, that for every nilpotent orbit  there exists a motivic function that coincides (up to a motivic constant) with the Shalika germ corresponding to that orbit on a definable neighbourhood of the origin. However, rescaling any given element of the Lie algebra so that it would fall into this neighbourhood presents a slight problem from the definable point of view, and so we shall address the full question of homogeneity elsewhere. 
It turns out that the existence of motivic functions that represent the Shalika germs in some small neighbourhood  of the origin 
is sufficient for the application we have in mind, namely, the uniform-in-$p$ bound on the normalized Shalika germs, which appears in Theorem \ref{thm:bound} below. 

\begin{theorem}\label{thm:motivic}
 Let $\fg =\fsl_{n}$ or $\fsp_{2n}$. Let $\mathcal N$ be the set of nilpotent elements in $\fg$.  Then
\begin{enumerate}
\item There exists a definable set
$\mathcal E$,  such that $\mathcal E_F$ is finite for all fields $F$ of sufficiently large residue characteristic, 
and a definable  function $h:\mathcal N\to \mathcal E$, such that
for every $d\in \mathcal E$, $h^{-1}(d)$ is an adjoint orbit,
 and each orbit appears 
 as the fibre of $h$.
\item 
There exist motivic functions $\mathbf{\Gamma}$ on
$\mathcal E\times\fg^{\reg}$ and $C$ on $\mathcal E$, and a constant $M>0$, such that for all local fields $F$ of residue characteristic greater than $M$,
for every $d\in {\mathcal E}_F$, the function 
$C^{-1}{\mathbf \Gamma}_F(d, \cdot)$ is a representative of the Shalika germ on 
$\fg^{\reg}$ corresponding to $d$, i.e., coincides with the Shalika germ on some neighbourhood of the origin.
\end{enumerate}
\end{theorem}

\begin{proof}
{\bf (1).} First, note that the set of nilpotent elements $\cN$ is, indeed, definable:
it is defined by the formula $X^n=0$ in $\fsl_n$ and by the formula $X^{2n}=0$ in $\fsp_{2n}$. 

For $\fg=\fsp_{2n}$, recall the parametrization of the nilpotent orbits from Theorem \ref{reps},
and let $\mathcal E$ be the set of pairs $(\lambda, \cQbar)$ as in Theorem 
\ref{thm:sp4-partitions}. 
Note that for each pair $(\lambda, \cQbar)\in \mathcal E_F$, there is an explicit representative $X_{(\lambda, \cQbar)}\in \cN_F$.  The definition of $X_{(\lambda, \cQbar)}$ involves only constant symbols in the extended Denef-Pas language, hence, the orbit of $X_{(\lambda, \cQbar)}$ is a definable set, and the map $h$ can be seen explicitly in Theorem \ref{reps}. 

 
For $\fg=\fsl_n$, the proof is essentially carried out in \cite{diwadkar:thesis}*{Section 6}; here we reinterpret it using the most recent version of motivic integration, and state it more generally. 
We are assuming that we are working with $\bG=\SL_n$, and $n$ is fixed.
There is a certain awkwardness to the proof caused by the fact that quotients, even by very nice definable equivalence relations, are not easy (sometimes impossible) to code in a first-order language.

Here we need to make a construction that allows us to handle the quotient $F^\times/(F^\times)^m$, where $F$ is the valued field, and so we use the union of the languages ${\ldp}_m$ defined above in \S \ref{sect:ldpo}, as $m$ runs over the divisors of $n$.  

More precisely, for every partition $\lambda$ of $n$, 
add the symbols for constants of the valued field sort
$d_{\lambda, 1}, \dots, d_{\lambda, m}$, where $m=\gcd(\lambda)$. 

Now, define the set $\mathcal E$ as the disjoint union over all partitions $\lambda$ of $n$, of sets $\mathcal E_{\lambda}$, defined as follows.
Given a partition $\lambda$, let $m=\gcd(\lambda)$ as above.
We have $m$ constant symbols corresponding to this partition,
$d_{\lambda, 1}, \dots, d_{\lambda, m}$, in the language. 
Recall the formulas $\phi_{\ell, m}$ from  (\ref{eq:philm}) in \S \ref{sect:ldpo}.  
With this value of $m$, exactly one of the formulas 
$\psi_{\ell,m}:=`\exists y_1, \dots, y_{\ell} \ \phi_{\ell, m}(y_1, \dots, y_\ell)\text{'}$ holds. 
If $\psi_{\ell, m}$ holds, we interpret the constant symbols  $d_{\lambda, 1},
\dots, d_{\lambda, \ell}$ as units  of the valued field such that 
$\phi_{\ell, m}(\ac(d_{\lambda, 1}), \dots, \ac(d_{\lambda, \ell}))$ holds.
Set the rest of 
the $d_i$ equal to 1.
(Note that this construction of the language is consistent with \S \ref{sect:ldpo}, but incorporates the union over $\lambda$.) 
Then let 
$$\mathcal E_{\lambda}:=\cup_{k=0}^{m-1}
\{\varpi^k d_{\lambda, 1}, \dots, \varpi^k d_{\lambda, \ell}\}.$$
This is a definable set since it consists of just the constant symbols in the language. 


Now, for every partition $\lambda$, and every $d\in \mathcal {E_\lambda}_F$, we  have the elements $X_d$ as defined in Proposition \ref{prop:partitions}. The disjoint union of the orbits of these elements is precisely 
$\mathcal N_F$. 

We will also need two observations about the set $\mathcal E$ for the proof of the second part of the theorem:
\begin{enumerate}
\item The set ${\mathcal E}^{\ge a}$ parametrizing orbits of dimension at least $a$ is definable for $a=0, \dots, n$.
\item The cardinality of the set $\mathcal E^{\ge a}$ is bounded independently of the field. 
\end{enumerate}

The first observation holds since the dimension of the orbit depends only on the partition $\lambda$; hence, the set $\mathcal E^{\ge a}$ is the disjoint union of $\mathcal E_\lambda$ over a prescribed set of partitions $\lambda$ that depends only on $a$. 
The second observation is immediate from the definition. 
Indeed, with the above notation, in the case of $\fsl_n$ the upper bound on the size of $\mathcal E^{\ge a}$ is given by $N(n,a):=\sum_{\lambda}\gcd(\lambda)^2$, where the sum is over the partitions 
$\lambda$ of $n$ that give rise to orbits of dimension at least $a$.
In the case of $\fsp_{2n}$, the statement is trivial since the number of orbits of a given dimension is field-independent from the start. 

Now we turn our attention to Part (2).  
 
{\bf (2).} 
First, let us discuss the restriction on the residue characteristic of $F$. 
We will be using \cite{CGH-2}*{Corollary 4.4} which states (in our case, without exponentials):
\begin{quote}
\em{Given a family of definable test functions $\{f_a\}_{a\in S}\subset \tf(\fg)$ with some definable set $S$, there exists a constant $M$ and a motivic function $h$ on $\fg\times S$ such that for all non-Archimedean
local fields $F$ of residue characteristic greater than $M$,
$$\mu_X(f_a)=h_F(X,a).$$
}
\end{quote}
Here we use this corollary with $S=\cE$. Recall from Part (1) that  there exists a constant $M_0$ such that for all $F$ with residue characteristic greater than $M_0$, for every $d\in \cE_F$, we have an element $X_d\in \fg(F)$ (an explicit matrix whose entries are constant symbols in the language ${\ldp}_m$ for some $m$), and the set 
$\{X_d\}_{d\in \cE_F}$ is a set of representatives of nilpotent orbits in $\fg(F)$. 
By the matching theorem 
(Theorem \ref{thm:sln-corresp} for the case of $\fsl_n$ and Theorem \ref{thm:sp2n-corresp} for 
$\fsp_{2n}$), there exists a unique pair 
$(\mathcal F, v)$ that corresponds to the orbit of $X_d$; in particular,
$X_d \in \fg_{\cF}$ and 
$v=X_d \mod \fg_{\cF}^+ \in V_{\cF}$. 
Note that 
$\fg_{\mathcal F}^+$ is an open compact subset of $\fg(F)$, and so is its translate $X_d+\fg_{\mathcal F}^+$. Let $f_d$ be the characteristic function of the coset 
$X_d+\fg_{\mathcal F}^+$. It is definable by \cite{CGH-2}*{Lemma 3.2}. 
Thus we have a family of definable test functions $\{f_d\}_{d\in \cE}$ indexed by the definable set 
$\cE$. Let $M$ be the maximum of $M_0$ and the constant from the statement of 
\cite{CGH-2}*{Corollary 4.4} quoted above, for this specific family. 
This will be the constant that appears as the restriction on the residue characteristic in our theorem.

Now we are ready to prove the statement of Part(2), for fields $F$ with residue characteristic greater than $M$.  The argument proceeds by downward induction on the dimension of the nilpotent orbit. 
The base case is an orbit of the top dimension, and the idea is to construct a definable test function whose support intersects only this orbit, which allows us to isolate the Shalika germ attached to the chosen orbit. 
For orbits of smaller dimension it is of course not possible to isolate a single Shalika germ, but it is possible to construct a definable test function whose support intersects only the given orbit and orbits of  strictly higher dimension.  This is where a theorem of Barbasch and Moy, refined by DeBacker and quoted above as Lemma \ref{lemma:debacker-nilp}, is needed, and this is how downward induction on the dimension proceeds.

Thus, for the base case, let $\lambda=(n)$ be the partition that gives rise to the orbits of the maximal dimension, which we denote by $a_{\mathrm{max}}$. Let $F$ be a local field with residue characteristic greater than $M$, and let $d\in \mathcal {E_{\lambda}}_F$. Let $X_d$ be the explicit representative of the corresponding orbit, as above. 
Let $f_d$ be the corresponding test function constructed above, {\em{i.e.}}, the characteristic function of the coset $X_d+\fg_{{\cF}^+}$, with $(\cF, v)$ the pair corresponding to $X_d$.

 By Lemma \ref{lemma:debacker-nilp},
the orbit of $X_d$ is the unique nilpotent orbit of minimal dimension intersecting $X_d+\fg_{\mathcal F}^+$; since there are no orbits of dimension greater than 
$a_{\mathrm{max}}$, in this case it means that the orbit of $X_d$ is the unique nilpotent orbit intersecting the support of the test function $f_d$. 
Thus, for the test function $f_d$ the Shalika germ expansion has only one term, namely
$$
\mu_X(f_d)=\Gamma_{X_d}(X)\mu_{X_d}(f_d),
$$
where the expansion holds for $X\in U_{f_d}\cap \fg(F)^\rss$, with $U_{f_d}$ some neighbourhood of the origin (which  depends on the test function $f_d$). 
By \cite{CGH-2}*{Corollary 4.4}, $\mu_{X_d}(f_d)$ is a motivic function of $d$ (that is, for a fixed $d$, a motivic constant, which we denote by $C(d)$), and $\mu_X(f_d)$ is a motivic function of 
$X$ and $d$. More precisely, there exists a motivic function ${\mathbf \Gamma}(X,d)$ such that ${\mathbf \Gamma}_F(X,d)=\mu_X(f_d)$. (Note that here we are using our definition of the constant $M$, and the assumption that the residue characteristic of $F$ is greater than $M$.) 
If necessary, we can shrink $U_d$ to make it definable. 
(Since $U_d$ is open, there exists a  lattice of the form $\fg_{x,r}$, {\em i.e.}, defined entirely by inequalities on the valuations of the entries of $X$, which is contained in $U_d$.)
This establishes the base case.

Now, let us assume the statement of the theorem holds for the orbits of dimension at least $a$ (where $a$ is an even integer).
Let $F$ be as above, and 
let $d\in \cE_F$ be a point such that the orbit of 
$X_d$ has dimension $a-2$. As above, there exists a unique pair 
$(\mathcal F, v)$ that corresponds to the orbit of $X_d$, ({\em i.e.} the orbit of $X_d$ is the unique orbit of minimal dimension intersecting $X_d+\fg_{\mathcal F}^+$). Let $f_d$ be the characteristic function of the coset
$X_d+\fg_{\mathcal F}^+$, as above.  Then the intersection of its support with $\cN_F$ is the union of its intersection with the orbit of $X_d$, and subsets of orbits of strictly higher dimension, {\em i.e.,} of dimension at least $a$. 
Then there exists a neighbourhood of $0$, which we will denote by $U_{f_d}$, such that for $X\in U_{f_d}$,
$$
\mu_X(f_d)=\Gamma_{X_d}(X)\mu_{X_d}(f_d)+\sum_{d'\in \cE_F^{\ge a}}\Gamma_{X_{d'}}(X)\mu_{X_{d'}}(f_d), $$
where the sum runs over the set of representatives of nilpotent orbits of dimension at least $a$. As in the base case, we can shrink $U_{f_d}$ to make it definable.
Then, for $X\in U_{f_d}$, we have 
\begin{equation}\label{eqn:mu-gamma}
\Gamma_{X_d}(X)\mu_d(f_d) = \mu_X(f_d)-\sum_{d'\in \cE_F^{\ge a}}\Gamma_{X_{d'}}(X)\mu_{X_{d'}}(f_d).
\end{equation} 
The right-hand side of \eqref{eqn:mu-gamma} is almost a motivic function. (Almost, because the Shalika germs corresponding to the orbits of greater dimension labelled by the points $d'$ are ratios of motivic functions and motivic constants.) More precisely, by the inductive assumption, for the Shalika germs occurring on the right-hand side of \eqref{eqn:mu-gamma}, we have 
$\Gamma_{X_{d'}}=C(d')^{-1}{\mathbf\Gamma}_F(X, d')$ in some definable neighbourhood $U_{d'}$ of the origin.
Let $U$ be the intersection of $U_{f_d}$ and all the $U_{d'}$ where $d'$ runs over 
$\mathcal E^{\ge a}$. 

Clearing denominators on both sides, we see that 
it remains only to prove that the product of the motivic constants 
$\prod_{\mathcal E^{\ge a}}C(d')$ is itself a motivic constant. 
In the case of $\fsp_{2n}$ this is clear, since  the indexing set in the product is field-independent, so we just have a fixed finite product of motivic constants.
In the case of $\fsl_n$, recall the sets 
$\mathcal E^{\ge a}$ defined in Part (1) above.
Let $P_a$ be the set of partitions $\lambda$ that give rise to the orbits of dimension at least $a$, so that $\cE^{\ge a}=\cup_{\lambda\in P_a}\cE_{\lambda}$. 
Recall from Part (1) that for each partition $\lambda$ we have the constant symbols  $d_{\lambda, 1}, \dots , d_{\lambda, m_{\lambda}}$, where 
$m_\lambda=\gcd(\lambda)$, and 
some of these constants specialize to  $1$ in a given field $F$, depending on the number of roots of unity in $F$. 
By definition, we have 
$$
\cE^{\ge a} = \bigsqcup_{\lambda\in P_a} \bigsqcup_{j=0}^{m_\lambda-1}
\{\varpi^j d_{\lambda, 1}, \dots, \varpi^j d_{\lambda, m_{\lambda}}\}.
$$
Let us define, for each $\lambda\in P_a$ and each $\ell, j$ with $1\le \ell, j\le m_\lambda$, a constant function   
$$\varphi_{\ell,j}:=\begin{cases}
1 &\text{ if } d_{\lambda, \ell}=1\\
C(d') &\text{ if } d'=\varpi^jd_{\lambda, \ell} \text{ with } d_{\lambda, \ell}\neq 1. 
\end{cases}$$
Then we can write 
$$\prod_{\mathcal E^{\ge a}}C(d')=\prod_{\lambda\in P_a}\left(\prod_{\ell, j=1}^{m_\lambda} \varphi_{\ell, j}
\prod_{j=0}^{m_\lambda-1} C(\varpi^j)\right).$$
Thus we have represented 
 $\prod_{\mathcal E^{\ge a}}C(d')$ as a product of a fixed ({\em i.e.,} field-independent) number of motivic constants, which proves it is itself a motivic constant, and  completes the proof of the induction step. 
We note for future reference that the motivic constants $C(d')$ are positive, since they are obtained as products of volumes of definable sets. 
\end{proof}


\subsection{Corollaries}
The first consequence of 
Theorem \ref{thm:motivic} is an alternate proof 
that Shalika germ expansion holds in large positive characteristic. (This is already known thanks to the work of DeBacker.) Indeed, 
if an equality of motivic functions holds in characteristic zero it holds in large positive characteristic by the Transfer Principle of Cluckers and Loeser, \cite{cluckers-loeser}.

However, the results of \cite{S-T}*{Appendix B} allow us to also prove a different type of corollary. 
First, we must recall some notation and a theorem of Harish-Chandra, which we quote here from \cite{kottwitz:clay}*{Theorem 17.9}. 

For a regular semisimple element $X$ of $\fg$, let 
$D(X)=\prod_{\alpha\in \Phi}|\alpha(X)|$ be the {\em Weyl discriminant} of $X$ (cf. \cite{kottwitz:clay}*{\S 7} for alternative definitions). 
For $d\in \mathcal E_F$, let
$\overline{\Gamma}_d(X):=|D(X)|^{1/2}\Gamma_d(X)$ be the normalized Shalika germ. Note that here we mean the canonical Shalika germ, not just the provisional Shalika germ considered above in Theorem \ref{thm:motivic}; thus it is a function defined on the set of all regular semisimple elements in $\fg(F)$. 
Let $\bT$ be a maximal torus of $\bG$, and $\ft$ its Lie algebra.  
Harish-Chandra proved the following result.
\begin{theorem}\label{thm:hc_locbound}
(\cite{hc:queens}, \cite{kottwitz:clay}*{Theorem 17.9})
Every normalized Shalika germ $\overline{\Gamma}_d$ is a locally bounded function on $\ft$.
(Here the local field $F$ is assumed to have characteristic zero.) 
\end{theorem}

Now, suppose we fix a definable compact subset in $\ft$ (or, more generally, a family of such definable compact subsets), so that we can vary the local field and still talk about the bound for the normalized Shalika germs, restricted to the specific set. 
We can ask how does the bound on $\overline{\Gamma}_d$ depend on the field $F$? (Or on the compact subset in question?) 
The next theorem answers both questions. 
Note that it is more convenient for us to talk about subsets of $\fg$ rather than subsets of $\ft$. Since there are finitely many 
conjugacy classes of tori (with an upper bound on their number independent of the field),
and since Shalika germs are conjugation-invariant, the local boundedness on $\fg^\rss$ follows. 

We would like to state a general result on the dependence of the bound for Shalika germs restricted to a compact subset of $\fg(F)$ on the field $F$ and on the compact subset in question. 
Typically, the compact subsets one is interested in are Moy-Prasad filtration lattices or other similar subsets. Since here we are working with explicitly-defined Lie algebras, we will say that the compact sets $K_n$ form a \emph{family of congruence lattices} if every set $K_n$ is defined by the formulas 
$$\ord(X_{ij})\ge \alpha_{ij}(n),$$ 
where $\alpha_{ij}$ are  $\Z$-valued Presburger-definable functions of the parameter $n$.
We believe  that most natural situations where questions about uniform bounds arise should satisfy this property. 

\begin{theorem}\label{thm:bound}
Let $\fg=\fsl_n$ or $\fsp_{2n}$. 
Let $K_n$ be a family of congruence lattices in $\fg$, indexed by a parameter $n\in \Z$. Then there exists a constant $M>0$ (that depends only on the formulas defining the family $K_n$), and definable $\Z$-valued functions $a$ and $b$ on $\mathcal E$, such that for all local fields $F$ with residue characteristic greater than $M$, for every $d\in \mathcal E_F$, 
\begin{equation}\label{eq:bound}
|\overline{\Gamma}_d (X)|\le q^{a(d)+b(d)n} \,\text{ for } X\in {K_n}_F\cap \fg^{\rss}(F),
\end{equation}
where $q$ is the cardinality of the residue field of $F$. 
\end{theorem}
\begin{proof} Note that since $\mathcal E_F$ is a finite set with an upper bound on its cardinality independent of $F$, an equivalent formulation would be to demand the existence of constants $a$ and $b$ such that (\ref{eq:bound}) holds, independently of $d$.   

Let $U_d$ be the neighbourhood given in the proof of Theorem \ref{thm:motivic}. Recall that this is a definable neighbourhood on which the Shalika germ expansion holds for the specific test function $f_d$ constructed in that proof.
Let $U$ be the intersection of  the definable sets $U_d$, for $d\in \mathcal E_F$. (It is non-empty and definable since the cardinality of $\mathcal E_F$ is bounded independently of $F$.) 
Then on the set $U$, by Theorem \ref{thm:motivic}, we have that 
$$
 \overline{\Gamma}_d (X)=|D(X)|^{1/2}C(d)^{-1}{\mathbf\Gamma_d}_F(X), 
 $$
where $C$ is a motivic function of $d$, and ${\mathbf\Gamma_d}$ is a motivic function of $d$ and $X$. 
The discriminant $D(X)$ is a definable function since it is a polynomial in the entries of $X$ (cf. \cite{kottwitz:clay}*{\S 7.5}), hence, $|D(X)|^{1/2}$ is a motivic function in our sense (cf. \cite{CGH-2}*{\S B.3.1}). 
Thus the right-hand side is the ratio of a motivic function of $X$ and $d$, and a motivic function $C(d)$. Since for each $d$, $C(d)$ is a positive motivic constant, {\em i.e.,} an element of $\Z[q^{-1}, (1-q^i)^{-1}, i>0]$, and since $\#\mathcal E_F$ is bounded independently of $F$, there exist constants $a_1, a_2\ge 0$ such that $q^{a_2}\ge C(d)\ge q^{-a_1}$ for all $d\in \mathcal E_F$.  
By Harish-Chandra's Theorem, quoted above as Theorem \ref{thm:hc_locbound}, for every local field $F$ of characteristic zero,  there exists a  constant ${A_d}_F$  (that depends on $F$) such that 
$$|\overline{\Gamma}_d(X)|\le {A_d}_F \text{ for } X\in U_F.$$
Therefore, for the fields $F$ of characteristic zero, we have an estimate for the motivic function $|D(X)|^{1/2}{\mathbf\Gamma_d}_F(X)$, given by
$$
|D(X)|^{1/2}|{\mathbf\Gamma_d}_F(X)|=C(d)|\overline{\Gamma}_d(X)|\le q^{a_2} {A_d}_F \text{ for } X\in U_F.
$$ 
Then by the uniform boundedness principle for motivic functions, \cite{S-T}*{Theorem B.6}, there exist constants $M$ and $a_d\in \Z$ such that for all local fields $F$ with residue characteristic greater than $M$, we have
$$
|D(X)|^{1/2}|{\mathbf\Gamma_d}_F(X)| \le q^{a_d} \text{ for } X\in U_F.
$$
Finally, we obtain, for $X\in U_F$, that 
$|\overline{\Gamma}_d(X)|\le q^{a_d+a_1}$. 
Let $a(d)=a_d+a_1$. Note that since all of the functions involved were motivic functions of $d$, this constant $a(d)$ depends definably on $d$ (though as noted at the beginning of the proof, this seems to be unimportant). 
Thus, we have proved the theorem for one specific definable open compact set -- namely, $U$. 

Now we can extend it to an arbitrary family of congruence lattices $\{K_n\}_{n>0}$ using the homogeneity of Shalika germs.
Namely, for every $n$ there exists an integer $j(n)$ such that $\varpi^{j(n)}K_n\subset U$, and $j(n)$ is a Presburger-definable function of $n$. Indeed, $U$ has to contain some congruence lattice defined by $\ord(X_{ij})\ge \beta_{ij}$, with some constants $\beta_{ij}\in \Z$. 
Then we can take $j(n):=\max_{i,j}(-\alpha_{ij}(n)+\beta_{ij})$, where $\alpha_{ij}$ are the functions in the definition of the family $\{K_n\}_{n>0}$. 
Then, by definition of the canonical Shalika germs $\Gamma_d$, we have 
$$\Gamma_d(X)=|\varpi^{j(n)}|^{m}\Gamma_d(\varpi^{2j(n)} X), $$ where $m$ is the dimension of the nilpotent orbit with parameter $d$. 
Note that by definition,  for $t\in F^\times$, we have $|D(tX)|=|t|^{\dim(\fg)-r}|D(X)|$, where $r$ is the rank of $\fg$ (cf. \cite{kottwitz:clay}*{(17.11.2)}). 
Putting this together, we obtain, for $X\in {K_n}_F$, 
$$
\begin{aligned}
\overline{\Gamma}_d(X)=|D(X)|^{1/2}\Gamma_d(X)=
|D(\varpi^{2j(n)}X)|^{1/2}|\varpi^{-j(n)}|^{\dim(\fg)-r}\,
|\varpi^{j(n)}|^{m}\,
\Gamma_d(\varpi^{2j(n)} X)\\
=q^{j(n)(\dim(\fg)-r-m)}\, \overline{\Gamma}_d(\varpi^{2j(n)}X).
\end{aligned}
$$
Therefore, we have 
$$|\overline{\Gamma}_d(X)|\le q^{a(d)+j(n)(\dim(\fg)-r-m)}.$$
It remains only to observe that since $j(n)$ is a Presburger function on $\Z$, it is piecewise-linear, and the statement follows.
\end{proof}

\end{document}